\documentclass[11pt]{article}

\usepackage[T1]{fontenc}
\usepackage{lmodern}
\usepackage{amsmath,amssymb,amsthm,mathtools}
\usepackage{mathrsfs}
\usepackage{microtype}
\usepackage{enumitem}
\usepackage{geometry}
\usepackage[
  backend=biber,
  style=alphabetic
]{biblatex}
\usepackage[hidelinks]{hyperref}

\allowdisplaybreaks[2]
\setlist{leftmargin=*,itemsep=0.25em,topsep=0.4em}

\newtheorem{theorem}{Theorem}[section]
\newtheorem{lemma}[theorem]{Lemma}
\newtheorem{proposition}[theorem]{Proposition}
\newtheorem{corollary}[theorem]{Corollary}
\newtheorem{conjecture}[theorem]{Conjecture}

\newtheorem{maintheorem}{Theorem}

\theoremstyle{definition}
\newtheorem{definition}[theorem]{Definition}
\newtheorem{example}[theorem]{Example}

\theoremstyle{remark}
\newtheorem{remark}[theorem]{Remark}

\newcommand{\CC}{\mathbb{C}}

\newcommand{\FF}{\mathcal F}

\title{Holonomy and Integrability of $p$-Closed Foliations}
\author{Yuting Liu}
\date{}

\begin{document}

\maketitle

\begin{abstract}
Motivated by the Ekedahl--Shepherd-Barron--Taylor conjecture, we study
arithmetic criteria for the global and local integrability of foliations in
characteristic zero. Our main result is a finite-holonomy theorem for the
regular parts of invariant prime divisors of corank one foliations which are
\(p\)-closed for almost all primes. As a consequence, we give a complete
proof for corank one foliations of the statement proposed by
Ekedahl--Shepherd-Barron--Taylor: a foliation on a smooth
projective variety which is \(p\)-closed for almost all primes is
algebraically integrable whenever it admits a compact leaf.

We also introduce a local analytic notion of \(p\)-closedness for foliation
germs on normal complex varieties and formulate local conjectures relating
it to meromorphic and holomorphic first integrals. We establish holomorphic
integrability for analytically \(p\)-closed foliations with canonical
singularities on smooth surfaces and on klt surface germs. Combining these
local results with the finite-holonomy theorem, we obtain holomorphic first
integrals for certain analytically \(p\)-closed non-dicritical surface germs
and for analytically \(p\)-closed simple singularities in arbitrary
dimension.

 Finally, we formulate a
connection-theoretic counterpart to local holomorphic integrability in terms
of flat meromorphic extensions of the Bott partial connection, and verify it
in dimension two.
\end{abstract}

\tableofcontents

\section{Introduction}
 The existence of first integrals is a fundamental problem in the theory of
foliations. When both the ambient space and the foliation are smooth, the
classical Frobenius theorem says that the foliation is locally induced by a
holomorphic submersion, and hence admits local holomorphic first integrals. In
the presence of singularities, however, the existence of such first integrals
is much more delicate.

For corank one foliations on smooth complex varieties, Malgrange proved that a local
holomorphic first integral exists provided that the singular locus has
codimension at least three \cite{Malgrange1976}. This result was subsequently
extended in several directions. Cerveau and Lins Neto treated singular ambient
spaces with mild singularities \cite{CerveauLinsNeto2008}, while related
results for foliations on normal threefolds were obtained by Spicer
\cite{Spicer2020} and by Spicer--Svaldi \cite{SpicerSvaldi2021}. Further
results on holomorphic first integrals for singular foliations appear, for
instance, in the work of Mattei--Moussu \cite{MatteiMoussu1980} and Moussu
\cite{Moussu1998}.

A global analogue is the problem of algebraic integrability. A foliation is
algebraically integrable if its general leaves are algebraic, or equivalently
if it is induced by a rational map. Ekedahl, Shepherd-Barron, and Taylor
proposed the following arithmetic criterion.

\begin{conjecture}[\cite{EkedahlShepherdBarronTaylor1999}]\label{conj}
Let \(X\) be a normal variety and let \(\mathcal F\) be a foliation on \(X\).
Then \(\mathcal F\) is algebraically integrable if and only if it is
\(p\)-closed modulo almost all primes.
\end{conjecture}

We recall the definition of \(p\)-closedness in
Definition~\ref{pclosedalgdef}. Conjecture~\ref{conj} remains open even for
surfaces. Bost proved it for regular foliations under the additional
assumption that the general leaves satisfy the Liouville property
\cite{Bost2001}; more recently, Druel obtained algebraicity results for
separatrices of reduced surface foliations \cite{Druel2025}. The conjecture is
closely related to the Grothendieck--Katz \(p\)-curvature conjecture
\cite{Katz1972} (See also Conjecture \ref{grothendieckkatz}): both express the principle that arithmetic integrability after
reduction modulo almost all primes should force geometric integrability in
characteristic zero. In fact, Conjecture~\ref{conj} implies the
Grothendieck--Katz conjecture; see \cite{Xu2025}.

The first main theorem of this paper is the following finiteness theorem for holonomy of an invariant divisor, which has applications in both algebraic and analytic integrability of foliations.

\begin{maintheorem}\label{theoremA}
Let \(X\) be a smooth variety and let \(\mathcal F\) be a corank one
foliation which is \(p\)-closed for almost all primes. Let \(D\subset X\) be
an irreducible \(\mathcal F\)-invariant prime divisor. Then every connected component of
\[
D^\circ:=D_{\mathrm{reg}}\setminus\operatorname{Sing}(\mathcal F)
\]
has finite holonomy group.
\end{maintheorem}

Theorem \ref{theoremA} implies the following corollary which was first stated in the unpublished manuscript \cite[Proposition~2.5]{EkedahlShepherdBarronTaylor1999}: 

\begin{corollary}\label{compactleafalgebraic}
Let \(X\) be a smooth projective variety and let \(\mathcal F\) be a corank
one foliation which is \(p\)-closed for almost all primes. If \(\mathcal F\)
has a compact leaf, then \(\mathcal F\) is algebraically integrable.
\end{corollary}

The proof in \cite{EkedahlShepherdBarronTaylor1999} considers the formal neighbourhood of the leaf and lifts
horizontal functions through successive infinitesimal neighbourhoods. The
higher-order step, however, asserts the vanishing of the relevant extension
classes after reduction without justifying this assertion. We will give a complete proof in this paper. 

We next turn to local analytic integrability. If a foliation is algebraically
integrable, then it is locally induced by a meromorphic first integral. Thus
Conjecture~\ref{conj} suggests that \(p\)-closedness modulo almost all primes
should force the existence of local meromorphic first integrals. The usual
arithmetic notion of \(p\)-closedness, however, depends on an algebraic model,
whereas the existence of a local first integral is invariant under analytic
changes of coordinates. To bridge this difference, we introduce in
Section~\ref{pclosedefinition} a local analytic notion of \(p\)-closedness.
Roughly speaking, a foliated germ is analytically \(p\)-closed if it is
analytically isomorphic to an algebraic germ which is \(p\)-closed modulo
almost all primes.

This leads to the following local conjectures.

\begin{conjecture}\label{conjecturemeromorphic}
Let \(X\) be a normal variety and let \(\mathcal F\) be a foliation on \(X\).
For every closed point \(x\in X\), the foliation \(\mathcal F\) is
analytically \(p\)-closed at \(x\) if and only if it is meromorphically
integrable at \(x\).
\end{conjecture}

\begin{conjecture}\label{conjectureholomorphiccanonical}
Let \(X\) be a normal variety and let \(\mathcal F\) be a foliation on \(X\).
Suppose that one of the following conditions holds:
\begin{enumerate}[label=\textup{(\arabic*)}]
\item \(\mathcal F\) has canonical singularities;
\item \(\mathcal F\) is non-dicritical.
\end{enumerate}
Then, for every closed point \(x\in X\), the foliation \(\mathcal F\) is
analytically \(p\)-closed at \(x\) if and only if it is analytically
integrable at \(x\).
\end{conjecture}

We refer to Definition~\ref{definitionpclosed} for analytic
\(p\)-closedness and to Definition~\ref{definitionanalyticintegral} for
analytic and meromorphic integrability. The ``if'' directions are immediate
from the definitions. The canonical cases established in the paper are
collected in the following theorem. They provide, in particular, the
primitive local models used in the applications of
Theorem~\ref{theoremA}.

\begin{theorem}\label{canonicalcases}
The following statements hold.
\begin{enumerate}[label=\textup{(\arabic*)}]
\item Let \(X\) be a smooth surface and let \(\mathcal F\) be a rank-one
foliation with canonical singularities. For every closed point \(x\in X\),
the foliation \(\mathcal F\) is analytically \(p\)-closed at \(x\) if and
only if it is analytically integrable at \(x\). Moreover, if
\(x\in\operatorname{Sing}(\mathcal F)\), then there are holomorphic
coordinates \(u,v\) centred at \(x\) and coprime positive integers \(a,b\)
such that
\[
T_{\mathcal F}
=
\mathcal O_X\left(
 bu\frac{\partial}{\partial u}
 -av\frac{\partial}{\partial v}
\right).
\]
Equivalently, \(u^av^b\) is a primitive holomorphic first integral.

\item Let \(X\) be a normal klt surface and let \(\mathcal F\) be a rank-one
foliation with canonical singularities. If \(\mathcal F\) is analytically
\(p\)-closed at a closed point \(x\), then it admits a holomorphic first
integral at \(x\).

\item Let \(X\) be smooth and let \(\mathcal F\) be a corank one foliation
with canonical singularities which is analytically \(p\)-closed at \(x\).
Suppose that either \(x\) is a general point of an irreducible component of
\(\operatorname{Sing}(\mathcal F)\) of codimension two, or every irreducible
component of \(\operatorname{Sing}(\mathcal F)\) containing \(x\) has
codimension at least three. Then \(\mathcal F\) admits a holomorphic first
integral at \(x\).
\end{enumerate}
\end{theorem}

Theorem \ref{theoremA} also yields local holomorphic first integrals in certain non-dicritical surface case where the singularity is mild in the following sense:

\begin{corollary}\label{nondicriticalsurfaceapplication}
Let \(X\) be a normal klt surface and let \(\mathcal F\) be a
non-dicritical foliation on \(X\). Suppose that \(\mathcal F\) is
analytically \(p\)-closed at \(x\in X\). Assume that \(\mathcal F\)
admits a reduction of singularities over \(x\) for which one of the
following conditions holds:
\begin{enumerate}[label=\textup{(\arabic*)}]
\item the reduced exceptional fibre is irreducible;
\item the reduction is non-branching, i.e. every irreducible exceptional component contains at most two
singular points of the transformed foliation.
\end{enumerate}
Then \(\mathcal F\) admits a holomorphic first integral at \(x\).
\end{corollary}

Here, by a \emph{reduction of singularities of
\((X,\mathcal F)\) over \(x\)} we mean, after possibly shrinking \(X\)
around \(x\), a proper birational morphism
\(
\pi:Y\longrightarrow X
\)
such that \(\pi\) is an isomorphism over \(X\setminus\{x\}\), the surface
\(Y\) is smooth, the reduced
exceptional curve
\(
E:=\pi^{-1}(x)_{\mathrm{red}}
\)
is a simple normal crossing divisor,
and the transformed foliation
\(
\mathcal G:=\pi^{-1}\mathcal F
\)
has only reduced singularities in a neighbourhood of \(E\). See also Definition \ref{definitionnonbranching}.

A second local application treats simple singularities in arbitrary
dimension.

\begin{corollary}\label{simpleintegrability}
Let \(X\) be a smooth variety and let \(\mathcal F\) be a corank one
foliation with a simple singularity at a closed point \(x\in X\). If
\(\mathcal F\) is analytically \(p\)-closed at \(x\), then it admits a
holomorphic first integral at \(x\).
\end{corollary}

We also construct formal first integrals for rank-one foliations with
canonical singularities; see Theorem~\ref{rankone1}. Finally, we study a
connection-theoretic counterpart of local holomorphic integrability. For a
regular foliation, the Bott connection is naturally defined only in the leaf
directions. Proposition~\ref{prop:regular-local-model} characterises local
holomorphic first integrals by the existence of a flat full connection with a
closed parallel conormal frame. This motivates the following singular
analogue.

\begin{conjecture}\label{conjbottconnection}
Let \(X\) be a smooth variety and let \(\mathcal F\) be a foliation with tangent
sheaf \(T_{\mathcal F}\subset T_X\) and canonical singularities. Let \(x\in X\)
be a closed point and suppose that \(\mathcal F\) is analytically \(p\)-closed at
\(x\). Then there exist an analytic neighbourhood \(U\) of \(x\) and a flat
holomorphic full connection defined on $U\setminus \operatorname{Sing}(\FF)$
\[
\nabla:
N_{\mathcal F}|_{U\setminus\operatorname{Sing}(\mathcal F)}
\longrightarrow
N_{\mathcal F}|_{U\setminus\operatorname{Sing}(\mathcal F)}
\otimes
\Omega^1_{U\setminus\operatorname{Sing}(\mathcal F)}
\]
whose restriction along \(\mathcal F\) is the Bott partial connection. Moreover, on
a Zariski dense open subset of \(U\setminus\operatorname{Sing}(\mathcal F)\),
the dual connection admits a frame of closed parallel holomorphic \(1\)-forms.
\end{conjecture}

\begin{maintheorem}\label{bottdimtwo}
The canonical-singularity case of
Conjecture~\ref{conjectureholomorphiccanonical} implies
Conjecture~\ref{conjbottconnection}. In particular,
Conjecture~\ref{conjbottconnection} holds in dimension two.
\end{maintheorem}

We conclude the introduction by outlining the proofs.

The proof of Theorem~\ref{theoremA} follows the formal-neighbourhood
strategy proposed in
\cite[Proposition~2.5]{EkedahlShepherdBarronTaylor1999}. On a dense affine open subset \(U\) of the invariant divisor, the Bott
connection on the conormal line bundle has vanishing \(p\)-curvature.
The rank-one case of the Grothendieck--Katz conjecture, due to the work
of Chudnovsky--Chudnovsky and Katz, then gives a finite étale cover
\(\nu:U'\to U\) on which the pull-back of the conormal line bundle
admits a nowhere-vanishing horizontal section \(z_1\). We lift \(z_1\)
through the infinitesimal neighbourhoods of \(U'\) order by order. At order \(m\), the obstruction is a closed
one-form \(\beta_m\); the associated principal
\(\mathbb G_a\)-connection has vanishing \(p\)-curvature, and
André's theorem on extensions of isotrivial connections implies that
\(\beta_m\) is exact. The resulting
horizontal coordinates are preserved by holonomy to every order, so
every loop lifting to \(U'\) has trivial holonomy. Since such loops form
a finite-index subgroup of \(\pi_1(U)\), the full holonomy group is
finite.

For Corollary~\ref{compactleafalgebraic}, the compact leaf is an algebraic
invariant divisor by Chow's theorem, and Theorem~\ref{theoremA} gives finite
holonomy. Reeb local stability then produces infinitely many compact leaves.
By Darboux-Jouanolou theorem, the foliation is algebraically integrable.
This gives the complete proof of the compact-leaf theorem for corank one foliations
anticipated in \cite[Proposition~2.5]{EkedahlShepherdBarronTaylor1999}.

The proof of Theorem~\ref{canonicalcases} begins with the smooth surface
case. McQuillan's formal diagonal normal form and results of Mattei--Moussu
produce a convergent primitive monomial first integral. The klt surface case
is reduced to the smooth case by a finite quasi-\'etale cover, while the
higher-dimensional statement follows from transverse surface sections and
the Mattei--Moussu extension theorem.

For Corollary~\ref{nondicriticalsurfaceapplication}, finite holonomy is
combined with the primitive monomial first integrals supplied by the reduced
surface models and with holonomy continuation along invariant exceptional
curves. In the irreducible case this gives a first integral near the whole
exceptional fibre. In the reducible case, the componentwise first integrals
are adjusted successively along the exceptional intersection tree; the
non-branching hypothesis guarantees the required primitivity at the
intersection points. Corollary~\ref{simpleintegrability} follows by applying
the one-component surface argument to a sufficiently general transverse
surface section and then using the Mattei--Moussu extension theorem.

The formal integrability result for rank-one foliations follows
from McQuillan's formal normal form and an explicit construction of monomial
invariants. For Theorem~\ref{bottdimtwo}, analytic integrability provides a
holomorphic map \(F=(F_1,\ldots,F_q)\). On the dense open set where
\(dF_1,\ldots,dF_q\) form a conormal frame, this frame is declared horizontal.
The resulting connection matrices are meromorphic and agree on overlaps
because they define the same connection on a dense open set.

\noindent\textbf{Acknowledgements.}
The author would like to express his gratitude to his PhD supervisor,
Paolo Cascini, for his constant support and many valuable suggestions.
He also thanks Calum Spicer for many fruitful discussions, and Federico
Bongiorno, Samuele Ciprietti, St\'ephane Druel, Dongchen Jiao, Jie Liu,
Roktim Mascharak, Jorge Vit\'orio Pereira, and Richard Thomas for helpful
conversations and comments. In particular, the author is grateful to
Jorge Vit\'orio Pereira and St\'ephane Druel for suggestions that led to
the present proof of Theorem~\ref{theoremA}. The author
acknowledges support from the Heilbronn Institute for Mathematical
Research and Imperial College London.

\section{Notation and preliminaries}\label{preli}
 We work over $\CC$ unless otherwise stated. By a
\emph{variety} we mean a separated integral scheme of finite type.
For a variety $X$, we write $X^{\mathrm{an}}$ for its analytification. Whenever an
argument is local in the analytic topology, we use the analogous terminology for
reduced irreducible complex analytic spaces.

\subsection{Foliations and first integrals}

\begin{definition}\label{definitionfoliation}
Let $X$ be a normal algebraic variety or a normal complex analytic space. A foliation
$\mathcal F$ of rank $r$ on $X$ is a coherent subsheaf
$T_{\mathcal F}\subset T_X$ of rank $r$ such that:
\begin{enumerate}[label=\textup{(\arabic*)}]
\item $T_{\mathcal F}$ is closed under the Lie bracket, equivalently the natural map
\[
\wedge^2T_{\mathcal F}\longrightarrow T_X/T_{\mathcal F},
\qquad v\wedge w\longmapsto [v,w]\bmod T_{\mathcal F},
\]
is zero;
\item $T_{\mathcal F}$ is saturated in $T_X$, equivalently
$T_X/T_{\mathcal F}$ is torsion-free.
\end{enumerate}
The corank of $\mathcal F$ is $\dim X-r$.
\end{definition}

\begin{remark}\label{uniqueopensubset}
A saturated subsheaf of a torsion-free sheaf is determined by its restriction to any
dense open subset. Consequently, a foliation on a normal algebraic variety
(resp. normal analytic space) is uniquely determined by its restriction to a nonempty
Zariski open subset (resp. a dense analytic open subset).
\end{remark}

\begin{definition}
Let $X$ be a normal variety and let $\mathcal F$ be a foliation of rank $r$. Its
canonical divisor $K_{\mathcal F}$ is the Weil divisor class determined by
\[
\mathcal O_X(-K_{\mathcal F})
=\det T_{\mathcal F}:=(\wedge^rT_{\mathcal F})^{**}.
\]
\end{definition}

\begin{definition}\label{singF}
Let $X$ be a normal variety and let $\mathcal F$ be a foliation of rank $r$ on
$X$. The inclusion $T_{\mathcal F}\subset T_X$ induces the Pfaff morphism
\[
\phi_{\mathcal F}:\Omega_X^{[r]}\longrightarrow\mathcal O_X(K_{\mathcal F}),
\qquad
\Omega_X^{[r]}:=(\Omega_X^r)^{**},
\]
obtained by taking the reflexive hull of the $r$-th exterior power of the
natural morphism $\Omega_X^{[1]}\to T_{\mathcal F}^*$. After twisting by
$\mathcal O_X(-K_{\mathcal F})$ and taking the reflexive hull, this gives
\[
\phi_{\mathcal F}':
\bigl(\Omega_X^{[r]}\otimes\mathcal O_X(-K_{\mathcal F})\bigr)^{**}
\longrightarrow\mathcal O_X.
\]
The \emph{singular locus} of $\mathcal F$, denoted by
$\operatorname{Sing}(\mathcal F)$, is the cosupport of the ideal sheaf
$\operatorname{Im}(\phi_{\mathcal F}')$.
\end{definition}

\begin{remark}\label{remarksingularlocusconvention}
We keep the singularities of the ambient variety and those of the foliation
separate. Thus $\operatorname{Sing}(X)$ is not automatically included in
$\operatorname{Sing}(\mathcal F)$. On $X_{\mathrm{reg}}$,
the preceding definition agrees with the usual one: it is the locus where
$T_{\mathcal F}$ fails to be a subbundle of $T_X$, equivalently where
$T_X/T_{\mathcal F}$ is not locally free.

We say that $\mathcal F$ is \emph{regular at $x$} if
\(
x\in X\setminus\operatorname{Sing}(\mathcal F),\)
and we call this open subset the regular locus of the foliated variety
$(X,\mathcal F)$. 
\end{remark}

\begin{definition}
Let $\mathcal F$ be a corank $q$ foliation on a normal variety $X$. Its normal sheaf is
\[
N_{\mathcal F}:=(T_X/T_{\mathcal F})^{**}.
\]
The natural morphism $N_{\mathcal F}^*\to\Omega_X^{[1]}$ induces a nonzero twisted
reflexive form
\[
\omega_{\mathcal F}\in
H^0\bigl(X,(\Omega_X^q\otimes\det N_{\mathcal F})^{**}\bigr),
\qquad \Omega_X^{[q]}:=(\Omega_X^q)^{**}.
\]
\end{definition}

\begin{remark}\label{foliationbyform}
On a dense open subset, the form $\omega_{\mathcal F}$ is decomposable: locally one may
write $\omega_{\mathcal F}=\omega_1\wedge\cdots\wedge\omega_q$. Its integrability is
equivalent to
\[
d\omega_i\wedge\omega_{\mathcal F}=0
\qquad\text{for }i=1,\ldots,q.
\]
Conversely, a nonzero twisted reflexive $q$-form which is generically decomposable and
integrable defines a foliation by saturating the kernel of the associated contraction
map. These constructions are inverse up to multiplication of the defining form by a
rational function.
\end{remark}

\begin{definition}
Let $\mathcal F$ be a corank $q$ foliation on a normal variety $X$, and let
$S\subset X$ be a sufficiently general normal hypersurface such that
$\omega_{\mathcal F}|_S\not\equiv0$. The restricted form
\[
\omega_{\mathcal F}|_S\in
H^0\bigl(S,(\Omega_S^q\otimes(\det N_{\mathcal F})|_S)^{**}\bigr)
\]
defines a corank $q$ foliation on $S$, denoted by $\mathcal F|_S$.
\end{definition}

\begin{definition}\label{pullbackfoliation}
Let $X$ and $Y$ be normal varieties, let $\varphi:X\dashrightarrow Y$ be a dominant
rational map, and let $\mathcal F$ be a corank $q$ foliation on $Y$. On a dense open
subset where $\varphi$ is a morphism and $\mathcal F$ is represented by a twisted
$q$-form $\omega$, the pull-back $\varphi^*\omega$ defines a foliation. Its unique
saturated extension to $X$ is the pull-back foliation, denoted by
$\varphi^{-1}\mathcal F$. If $\varphi$ is birational, this is also called the strict
transform of $\mathcal F$.
\end{definition}

\begin{definition}
Let $X$ be a normal variety and let $\mathcal F$ be a foliation on
$X$. A subvariety $Z\subset X$ is $\mathcal F$-invariant if every local section
$\partial$ of $T_{\mathcal F}$ preserves its ideal sheaf:
\(
\partial(I_Z)\subset I_Z.
\)
An irreducible invariant analytic germ $S$ of dimension
$\operatorname{rank}(\mathcal F)$ is called a \emph{separatrix} if
\(
S_{\mathrm{reg}}\cap X_{\mathrm{reg}}
\setminus\operatorname{Sing}(\mathcal F)\neq\varnothing.
\)

A \emph{leaf} of $\mathcal F$ is a maximal connected immersed complex
submanifold of
\(
X_{\mathrm{reg}}\setminus\operatorname{Sing}(\mathcal F)
\)
tangent to $T_{\mathcal F}$.
\end{definition}

\begin{definition}
Let $f:X\dashrightarrow Y$ be a dominant rational map between normal varieties. On the
dense open subset where $f$ is a morphism and is smooth over its image, the relative
tangent sheaf
\[
T_{X/Y}=\ker(df:T_X\to f^*T_Y)
\]
defines a foliation. Its saturated extension to $X$ is the foliation induced by $f$.
A foliation is algebraically integrable if it is induced by a rational map.
\end{definition}

\begin{remark}
A foliation is algebraically integrable if and only if its general leaf is algebraic.
\end{remark}

\begin{definition}\label{definitionfirstintegral}
Let $X$ be a normal variety of dimension $n$, let $\mathcal F$ be a foliation of rank
$r$, and let $x\in X$. A holomorphic first integral of $\mathcal F$ at $x$ is a
holomorphic map
\[
f:U\longrightarrow\mathbb C^{n-r}
\]
on an analytic neighbourhood $U\subset X^{\mathrm{an}}$ of $x$ which induces
$\mathcal F^{\mathrm{an}}|_U$.

A meromorphic first integral of $\mathcal F$ at $x$ is a meromorphic map
$f:U\dashrightarrow\mathbb C^{n-r}$ such that, on a dense open subset where $f$ is
holomorphic, it is a holomorphic first integral of $\mathcal F$.

When $n-r=1$, a holomorphic first integral is called \emph{primitive} if its general
fibres are connected. Equivalently, after shrinking, it is not a nontrivial composition
$\varphi\circ g$ with a one-variable germ $\varphi$ of multiplicity greater than one.
Every holomorphic first integral admits a primitive factor by local Stein
factorisation.
\end{definition}

\subsection{Singularities of foliations}

\begin{definition}
Let $X$ be normal and let $\mathcal F$ be a foliation such that $K_{\mathcal F}$ is
$\mathbb Q$-Cartier. If $\pi:Y\to X$ is a birational morphism and
$\widetilde{\mathcal F}:=\pi^{-1}\mathcal F$, then
\[
K_{\widetilde{\mathcal F}}
\sim_{\mathbb Q}
\pi^*K_{\mathcal F}+\sum_E a(E,\mathcal F)E,
\]
where the sum runs over the $\pi$-exceptional prime divisors. The coefficient
$a(E,\mathcal F)$ is the discrepancy of $\mathcal F$ along $E$.
\end{definition}

\begin{definition}
In the preceding setting, let $E$ be a prime divisor over $X$ and define
\[
\epsilon(E)=
\begin{cases}
0,&\text{if }E\text{ is invariant under the induced foliation},\\
1,&\text{otherwise}.
\end{cases}
\]
The foliation $\mathcal F$ is terminal, canonical, log terminal, or log canonical if,
respectively,
\[
a(E,\mathcal F)>0,\qquad
a(E,\mathcal F)\ge0,\qquad
a(E,\mathcal F)>-\epsilon(E),\qquad
a(E,\mathcal F)\ge-\epsilon(E)
\]
for every divisor $E$ exceptional over $X$.
\end{definition}

\begin{definition}[{\cite{CasciniSpicer2021}}]\label{definitionnondicritical}
Let $X$ be normal and let $\mathcal F$ be a foliation on $X$. We say that $\mathcal F$
is non-dicritical if, for every closed point $p\in X$ and every proper birational
morphism $\pi:Y\to X$ for which $\pi^{-1}(p)$ is a divisor, every irreducible
component of $\pi^{-1}(p)$ is invariant under $\pi^{-1}\mathcal F$.
\end{definition}

\begin{definition}[{\cite{Cano2004}}]\label{definitionsimple}
Let $X$ be a smooth variety of dimension $n$, let $\mathcal F$ be a corank one
foliation, and let $x\in X$ be a closed point. We say that $\mathcal F$ has a
\emph{simple singularity} at $x$ if there are formal coordinates
$x_1,\ldots,x_n$ centred at $x$ and an integer $2\le r\le n$ such that
$N_{\mathcal F}^*$ is generated by a formal $1$-form of one of the following types.

\begin{enumerate}[label=\textup{(\arabic*)}]
\item There are $\lambda_1,\ldots,\lambda_r\in\mathbb C^*$ such that
\[
\sum_{i=1}^r a_i\lambda_i\neq0
\]
for every nonzero $(a_1,\ldots,a_r)\in\mathbb Z_{\ge0}^r$, and
\[
\omega=x_1\cdots x_r\sum_{i=1}^r\lambda_i\frac{dx_i}{x_i}.
\]

\item There are an integer $1\le k\le r$, positive integers
$p_1,\ldots,p_k$ without common factor, a formal series
$\psi\in t\mathbb C[[t]]$, and constants $\alpha_2,\ldots,\alpha_r\in\mathbb C$
such that
\[
\sum_{i=k+1}^r a_i\alpha_i\neq0
\]
for every nonzero $(a_{k+1},\ldots,a_r)\in\mathbb Z_{\ge0}^{r-k}$, and
\[
\omega=x_1\cdots x_r\left(
\sum_{i=1}^k p_i\frac{dx_i}{x_i}
+
\psi(x_1^{p_1}\cdots x_k^{p_k})
\sum_{i=2}^r\alpha_i\frac{dx_i}{x_i}
\right).
\]
When $k=r$, the displayed non-resonance condition is empty.
\end{enumerate}

The integer $r$ is called the \emph{dimensional type} of the singularity.
\end{definition}

\subsection{Foliations on surfaces}

\begin{definition}\label{surfacedefinition}
Let $X$ be a smooth surface, let $\mathcal F$ be a foliation on $X$, and let
$x\in\operatorname{Sing}(\mathcal F)$. Choose a local vector field $v$ generating
$\mathcal F$ near $x$, and let $\lambda_1,\lambda_2$ be the eigenvalues of the linear
part $Dv(x)$. The singularity is reduced if at least one eigenvalue is nonzero and,
when both are nonzero, $\lambda_1/\lambda_2\notin\mathbb Q_{>0}$.
\end{definition}

\begin{remark}
The quotient $\lambda_1/\lambda_2$ is unchanged if $v$ is multiplied by a nowhere
vanishing function, up to replacing it by its inverse. We call it the eigenvalue of the
foliation at $x$.
\end{remark}

\begin{definition}
A reduced singularity is non-degenerate if both eigenvalues are nonzero; otherwise it
is a saddle-node.
\end{definition}

\begin{theorem}[Poincaré linearisation theorem]\label{Poincare}
Let $X$ be a smooth surface, let $\mathcal F$ be a foliation on $X$, and let
$x\in\operatorname{Sing}(\mathcal F)$ be a non-degenerate singular point. Suppose
that $\mathcal F$ is formally linearizable and that the quotient of the two
nonzero eigenvalues of a local generator does not belong to
$\mathbb R_{\le0}$. Then the formal linearisation is convergent. In particular,
there are analytic coordinates $(z,w)$ centred at $x$ in which $\mathcal F$ is
generated by
\[
\lambda_1 z\frac{\partial}{\partial z}
+
\lambda_2 w\frac{\partial}{\partial w}.
\]
\end{theorem}

\begin{theorem}[{\cite{Seidenberg1968}}]\label{Seidenberg}
Let $X$ be a smooth surface and let $\mathcal F$ be a foliation on $X$. For every
$p\in\operatorname{Sing}(\mathcal F)$, there is a sequence of point blow-ups
$\pi:\widetilde X\to X$ over $p$ such that $\pi^{-1}\mathcal F$ has only reduced
singularities along $\pi^{-1}(p)$.
\end{theorem}

\section{\texorpdfstring{$p$-closed}{p-closed} foliations}\label{pclosedefinition}

We first recall $p$-closedness in positive characteristic and its arithmetic analogue
in characteristic zero. We then introduce the local analytic notion used throughout
the paper.

\subsection{\texorpdfstring{Arithmetic $p$-closedness}{Arithmetic p-closedness}}

\begin{definition}
Let $k$ be a field of characteristic $p>0$, let $X$ be a normal variety over $k$, and
let $\mathcal F$ be a foliation on $X$. We say that $\mathcal F$ is
\emph{$p$-closed} if, for every open subset $U\subset X$ and every local vector field
$v\in H^0(U,T_{\mathcal F})$, the $p$-fold iterate $v^p$, viewed as a derivation, is again
a section of $T_{\mathcal F}$ over $U$. In characteristic $p$, the operator $v^p$ is again
a derivation.
\end{definition}

\begin{definition}[Foliation $p$-curvature]\label{pcurvature}
The $p$-th power operation induces an $\mathcal O_X$-linear morphism
\[
\psi_{\mathcal F}:\operatorname{Frob}^*T_{\mathcal F}
\longrightarrow (T_X/T_{\mathcal F})^{**}=N_{\mathcal F},
\qquad
v\longmapsto v^p\bmod T_{\mathcal F},
\]
where $\operatorname{Frob}$ is the absolute Frobenius morphism. We call
$\psi_{\mathcal F}$ the $p$-curvature of $\mathcal F$. The foliation is $p$-closed
if and only if $\psi_{\mathcal F}=0$.
\end{definition}

\begin{lemma}\label{pcloseopenpositivechar}
Let $X$ be a normal variety over a field of characteristic $p>0$, let $\mathcal F$ be
a foliation, and let $U\subset X$ be a nonempty open subset. Then $\mathcal F$ is
$p$-closed if and only if $\mathcal F|_U$ is $p$-closed.
\end{lemma}

\begin{proof}
Only the ``if'' direction requires proof. By Definition~\ref{pcurvature}, the image of
$\psi_{\mathcal F}$ is supported on $X\setminus U$. Since $N_{\mathcal F}$ is
torsion-free and $X$ is integral, a nonzero coherent subsheaf of
$N_{\mathcal F}$ cannot be supported on a proper closed subset. Hence
$\psi_{\mathcal F}=0$.
\end{proof}

\begin{definition}\label{pclosedalgdef}
Let \(K\) be a field finitely generated over \(\mathbb Q\), let \(X\) be a normal
variety over \(K\), and let \(\mathcal F\) be a foliation on \(X\). We say that
\(\mathcal F\) is \emph{\(p\)-closed for almost all primes} if there exist:
\begin{enumerate}[label=\textup{(\arabic*)}]
\item a finitely generated integral \(\mathbb Z\)-subalgebra \(R\subset K\) with
\(\operatorname{Frac}(R)=K\);
\item a flat morphism of finite type
\(
\mathscr X\longrightarrow S:=\operatorname{Spec}R
\)
together with an identification of its generic fibre with \(X\);
\item a coherent subsheaf
\(
\mathscr F\subset T_{\mathscr X/S},
\)
flat over \(S\), whose generic fibre is identified with \(T_{\mathcal F}\);
\end{enumerate}
such that there is a nonempty open subset \(S^\circ\subset S\) for which, for every
closed point \(s\in S^\circ\), the fibre
\(
\mathscr F_s\subset T_{\mathscr X_s/\kappa(s)}
\)
is a \(p(s)\)-closed foliation. Here
\(
\kappa(s):=\mathcal O_{S,s}/\mathfrak m_s\)
and \(
p(s):=\operatorname{char}\kappa(s)
\)
are the residue field and the residue characteristic of \(s\), respectively. 

For simplicity, we call a flat morphism as in (2) a spreading out of $(X,\FF)$ in the following.
\end{definition}

\begin{lemma}\label{independenceofspreadingout}
The property in Definition~\ref{pclosedalgdef} is independent of the chosen
finitely generated \(\mathbb Z\)-subalgebra and spreading out. More precisely, if
one spreading out has \(p\)-closed fibres over a nonempty open subset of its base,
then every other spreading out has this property after shrinking its base.
\end{lemma}

\begin{proof}
Let
\(
(R_1,\mathscr X_1,\mathscr F_1)
\) and
\(
(R_2,\mathscr X_2,\mathscr F_2)
\)
be two spreadings out of \((X,\mathcal F)\). Let \(R_3\subset K\) be the
\(\mathbb Z\)-subalgebra generated by \(R_1\) and \(R_2\). Since all the data are
of finite presentation and the generic fibres are identified with the same foliated
variety, after localising \(R_3\) the two base-changed models are isomorphic as
foliated models.

Suppose that the first model has \(p\)-closed fibres over a nonempty open subset of
\(\operatorname{Spec}R_1\). Its inverse image in \(\operatorname{Spec}R_3\) is
nonempty and open, and the base-changed fibres are still \(p\)-closed: the
\(p\)-curvature morphism commutes with extension of the residue field. Through the
above isomorphism, the same is true for the base change of the second model over a
nonempty open subset \(T\subset\operatorname{Spec}R_3\).

The morphism \(T\to\operatorname{Spec}R_2\) is dominant and of finite type, so its
image contains a nonempty open subset \(S_2^\circ\). Let \(s_2\) be a closed point
of \(S_2^\circ\). There is a closed point \(s_3\in T\) over \(s_2\). The residue
field extension \(\kappa(s_2)\subset\kappa(s_3)\) is faithfully flat, and the
\(p\)-curvature of \(\mathscr F_{2,s_2}\) becomes zero after this extension.
Consequently it is already zero over \(\kappa(s_2)\). Thus
\(\mathscr F_{2,s_2}\) is \(p(s_2)\)-closed. This proves the assertion.
\end{proof}

\begin{remark}
The phrase ``for almost all primes'' refers to the closed points of a nonempty open
subset of \(\operatorname{Spec}R\), not merely to rational prime numbers. Different
closed points may have the same residue characteristic. If \(K\) is a number field
and \(R\) is a localisation of its ring of integers, the closed points are the nonzero
prime ideals not inverted in \(R\), and \(p(s)\) is their residue characteristic.
\end{remark}

\begin{definition}\label{pcloseddefc}
Let $X$ be a normal variety over $\mathbb C$, and let $\mathcal F$ be a foliation on
$X$. We say that $\mathcal F$ is \emph{$p$-closed for almost all primes} if there
exist a subfield $K\subset\mathbb C$ finitely generated over $\mathbb Q$, a normal
variety $X_K$ over $K$, and a foliation $\mathcal F_K$ on $X_K$, such that
\(
X_K\times_K\mathbb C\simeq X,\)
and \(
\mathcal F_K\otimes_K\mathbb C\simeq\mathcal F,
\)
and $\mathcal F_K$ is $p$-closed for almost all primes in the sense of
Definition~\ref{pclosedalgdef}.
\end{definition}

\begin{lemma}\label{pclosedopencondition}
Let $X$ be a normal variety over $\mathbb C$, let $\mathcal F$ be a foliation, and
let $U\subset X$ be a nonempty Zariski open subset. Then $\mathcal F$ is $p$-closed
for almost all primes if and only if $\mathcal F|_U$ is $p$-closed for almost all
primes.
\end{lemma}

\begin{proof}
After choosing a model over a finitely generated $\mathbb Z$-algebra, the assertion
follows from Lemma~\ref{pcloseopenpositivechar} on the closed fibres, after possibly
shrinking the base.
\end{proof}

\subsection{Flat connections and the Grothendieck--Katz conjecture}

The finite-holonomy argument in Section~\ref{holonomysection} uses the analogous notion of $p$-curvature
for flat connections.

\begin{definition}[Connection \(p\)-curvature]\label{connectionpcurvature}
Let \(k\) be a field of characteristic \(p>0\), let \(B\) be a smooth
\(k\)-variety, and let
\(
F_B:B\longrightarrow B
\)
be the absolute Frobenius morphism. Let \(\mathcal E\) be a vector bundle on
\(B\), and let
\[
\nabla:\mathcal E\longrightarrow\mathcal E\otimes\Omega^1_{B/k}
\]
be a flat algebraic connection. For a local vector field
\(v\in T_{B/k}\), let \(\nabla_v\) denote covariant differentiation along
\(v\), and put \(v^{[p]}:=v^p\), the \(p\)-fold composition of \(v\) as a
derivation of \(\mathcal O_B\). The operator
\[
\psi_\nabla(v):=\nabla_v^p-\nabla_{v^{[p]}}
\]
is \(\mathcal O_B\)-linear on \(\mathcal E\), and the assignment
\(v\mapsto\psi_\nabla(v)\) is \(p\)-linear. It therefore defines an
\(\mathcal O_B\)-linear morphism
\[
\psi_\nabla:F_B^*T_{B/k}
\longrightarrow
\mathcal{E}nd_{\mathcal O_B}(\mathcal E),
\]
called the \emph{\(p\)-curvature} of \(\nabla\).

More generally, let \(G\) be a smooth algebraic group over \(k\), and
let
\(
\pi:P\longrightarrow B
\)
be a right principal \(G\)-bundle. An algebraic connection on \(P\) is
a \(G\)-invariant subbundle \(H\subset T_{P/k}\) such that
\[
d\pi|_H:H\stackrel{\simeq}{\longrightarrow}\pi^*T_{B/k}.
\]
It is \emph{flat} if \(H\) is closed under the Lie bracket.

Assume that \(H\) is flat. If \(v\) is a local section of
\(T_{B/k}\), denote its unique horizontal lift by \(\widetilde v\).
Since
\[
d\pi\bigl(\widetilde v^{[p]}\bigr)
=
v^{[p]}
=
d\pi\bigl(\widetilde{\,v^{[p]}\,}\bigr),
\]
the difference
\(
\widetilde v^{[p]}-\widetilde{\,v^{[p]}\,}
\)
is vertical. It is also \(G\)-invariant, and hence determines a local
section of
\[
\operatorname{ad}(P)
:=
P\times^G\operatorname{Lie}(G).
\]
The map
\(
v\longmapsto
\widetilde v^{[p]}-\widetilde{\,v^{[p]}\,}
\)
is \(p\)-linear and therefore defines an
\(\mathcal O_B\)-linear morphism
\[
\psi_H:
F_B^*T_{B/k}
\longrightarrow
\operatorname{ad}(P),
\]
called the \emph{\(p\)-curvature} of the principal connection.
\end{definition}

\begin{conjecture}[Grothendieck--Katz, \cite{Katz1972}]
\label{grothendieckkatz}
Let \(K\) be a field of characteristic zero, let \(B\) be a smooth
geometrically connected \(K\)-variety, and let
\((\mathcal E,\nabla)\) be a flat algebraic connection on \(B\).
Suppose that, after spreading out, its reductions have vanishing
\(p\)-curvature for almost all primes. Then
\((\mathcal E,\nabla)\) becomes trivial after a finite étale cover.

When \(K\subset\mathbb C\), this is equivalent to saying that the
connection has finite monodromy.
\end{conjecture}

The converse implication is immediate. We shall use two known cases of
the forward implication. The first is the rank-one case.

\begin{theorem}[\cite{ChudnovskyChudnovsky1985}, \cite{Katz1982}]
\label{rankoneGK}
Let \(K\subset\mathbb C\) be a field finitely generated over
\(\mathbb Q\), let \(B\) be a smooth geometrically connected
\(K\)-variety, and let \((\mathcal L,\nabla)\) be a line bundle with a
flat algebraic connection on \(B\). Suppose that, after spreading out,
its reductions have vanishing \(p\)-curvature for almost all primes.
Then the induced connection on \(B_{\mathbb C}\) has finite monodromy.
Equivalently, after passing to a connected finite étale cover of
\(B_{\mathbb C}\), the pull-back of \(\mathcal L\) admits a
nowhere-vanishing horizontal section.
\end{theorem}

\begin{proof}
The result for rank-one connections on algebraic curves is
\cite[Theorem~8.1]{ChudnovskyChudnovsky1985}. The general case follows
from Katz's reduction to curves
\cite[proof of Theorem~10.5]{Katz1982}.
\end{proof}

Recall that a flat algebraic connection is called \emph{isotrivial} if
it becomes trivial after a finite étale cover. For the higher-order
lifting argument, we use the following theorem of André.

\begin{theorem}[{\cite[Corollary~4.3.4 and
Corollary~7.1.3]{Andre2004}}]
\label{andreextension}
Let \(K\) be a field of characteristic zero, let \(B\) be a smooth
geometrically connected \(K\)-variety, and let
\[
0\longrightarrow
(\mathcal E',\nabla')
\longrightarrow
(\mathcal E,\nabla)
\longrightarrow
(\mathcal E'',\nabla'')
\longrightarrow0
\]
be an exact sequence of flat algebraic connections on \(B\).
Suppose that \((\mathcal E',\nabla')\) and
\((\mathcal E'',\nabla'')\) are isotrivial. If, after spreading out,
the reductions of \((\mathcal E,\nabla)\) have vanishing
\(p\)-curvature for almost all primes, then
\((\mathcal E,\nabla)\) is isotrivial.
\end{theorem}
For comparison, over number fields Bost proved the following more
general principal-bundle statement.

\begin{theorem}[{\cite[Theorem~2.9]{Bost2001}}]
\label{bostsolvable}
Let \(K\) be a number field, let \(G\) be an algebraic group over \(K\)
whose neutral component is solvable, and let \(P\to B\) be a principal
\(G\)-bundle over a smooth connected \(K\)-variety, equipped with a
flat connection. Then the reductions of the connection have vanishing
\(p\)-curvature for almost all primes if and only if the bundle with
connection becomes trivial after a finite étale cover of \(B\).
\end{theorem}

\subsection{Examples}

\begin{example}\label{example1}
Let \(X=\mathbb C^2\), and let \(\mathcal F\) be the foliation generated by
\(
v=x\partial_x+y\partial_y.
\)
Then \(X\) and \(\mathcal F\) are defined over \(\mathbb Q\). Taking the standard model
\(\operatorname{Spec}\mathbb Z[x,y]\), for every prime number \(p\) we have
\[
(x\partial_x+y\partial_y)^p=x\partial_x+y\partial_y
\]
in characteristic $p$.
Hence \(\mathcal F\) is \(p\)-closed for almost all primes. The separatrices through the
origin are the lines \(ax-by=0\), with \([a:b]\in\mathbb P^1\), and hence are algebraic.
\end{example}

\begin{example}
Let \(X=\mathbb C^2\), and let \(\mathcal F\) be generated by
\(
v=x^2\partial_x+y\partial_y.
\)
With respect to the standard model over \(\mathbb Z\), in characteristic \(p\) one has
\[
(x^2\partial_x+y\partial_y)^p=y\partial_y.
\]
Thus \(v^p\) is not contained in the rank-one sheaf generated by \(v\), and hence
\(\mathcal F\) is not \(p\)-closed for almost all primes. The separatrices through the
origin are given by \(x=0\) and \(y=0\), while the other leaves are analytic curves of
the form
\(
y=Ce^{-1/x},
\)
\(C\in\mathbb C,\)
and are not algebraic.
\end{example}

\begin{example}
Let \(X=\mathbb G_m\times\mathbb G_m\) be a complex algebraic torus. A torus-invariant
rank-one foliation is generated by a vector field of the form
\(
v=x\partial_x+\lambda y\partial_y
\)
for some \(\lambda\in\mathbb C\), after rescaling. If \(\lambda\in\mathbb Q\), then after
choosing a model over \(\mathbb Z\) away from the primes dividing the denominator of
\(\lambda\), we have
\[
v^p=x\partial_x+\lambda^p y\partial_y
=x\partial_x+\lambda y\partial_y
\]
in characteristic \(p\). Thus \(\mathcal F\) is \(p\)-closed for almost all primes.

Conversely, if the foliation is \(p\)-closed for almost all primes, then the reductions
of \(v^p\) are proportional to the reductions of \(v\) for almost all primes. This forces
\(\lambda^p=\lambda\) modulo almost all primes, hence \(\lambda\in\mathbb Q\). Therefore
\(\mathcal F\) is \(p\)-closed for almost all primes if and only if \(\lambda\in\mathbb Q\).

Moreover, \(\mathcal F\) is algebraically integrable if and only if \(\lambda\in\mathbb Q\).
Indeed, if \(\lambda=m/n\) with \(m,n\in\mathbb Z\), then \(x^m y^{-n}\) is a rational
first integral. If \(\lambda\notin\mathbb Q\), the leaves are not algebraic.
\end{example}

\subsection{\texorpdfstring{Local analytic $p$-closedness}{Local analytic p-closedness}}

We now introduce the local analytic notion of $p$-closedness used in the main results.

\begin{definition}\label{definitionpclosed}
Let \(X\) be a normal complex variety and let \(\mathcal F\) be a foliation on \(X\).
Let \(x\in X\) be a closed point. We say that \(\mathcal F\) is
\emph{analytically \(p\)-closed at \(x\)} if there exist a normal complex variety \(Y\),
a closed point \(y\in Y\), and a foliation \(\mathcal G\) on \(Y\), such that:
\begin{enumerate}[label=\textup{(\arabic*)}]
\item the germs \((X,\mathcal F,x)\) and \((Y,\mathcal G,y)\) are analytically
isomorphic, i.e. there exist analytic neighbourhoods \(U\subset X^{\mathrm{an}}\) of
\(x\) and \(V\subset Y^{\mathrm{an}}\) of \(y\), together with an analytic isomorphism
\(U\simeq V\) identifying \(\mathcal F^{\mathrm{an}}|_U\) with
\(\mathcal G^{\mathrm{an}}|_V\);
\item \(\mathcal G\) is \(p\)-closed for almost all primes.
\end{enumerate}
We call \((Y,\mathcal G,y)\) a local \(p\)-closed model of
\((X,\mathcal F,x)\).
\end{definition}

\begin{example}\label{ex1}
Let \(X=\mathbb C^2\), and let \(\mathcal F\) be the foliation generated by
\(x\partial_x+\partial_y\). Since \(\mathcal F\) is regular, by the holomorphic Frobenius
theorem, for every point \(q\in \mathbb C^2\) there are analytic coordinates
\((x',y')\) centred at \(q\) in which \(\mathcal F\) is generated by \(\partial_{x'}\).
Let \((Y,\mathcal G)=(\mathbb C^2,\langle \partial_x\rangle)\). In characteristic \(p\),
one has \(\partial_x^p=0\). Hence \(\mathcal G\) is \(p\)-closed for all primes \(p\).
Therefore \(\mathcal F\) is analytically \(p\)-closed at every closed point of \(X\).
\end{example}

\begin{example}\label{ex2}
Let \(\mathcal F\) be the foliation on \(\mathbb C^2\) generated by
\(x^2\partial_x+y\partial_y\). Away from the origin, the foliation is regular, and hence
is analytically \(p\)-closed by Example~\ref{ex1}. The origin is a saddle-node.
By Lemma~\ref{pclosednotsaddle} below, \(\mathcal F\) is not analytically \(p\)-closed at
the origin.
\end{example}

\begin{definition}\label{definitionanalyticintegral}
Let \(X\) be a normal complex variety and let \(\mathcal F\) be a foliation of corank
\(q\) on \(X\). Let \(x\in X\) be a closed point.

We say that \(\mathcal F\) is \emph{analytically integrable at \(x\)} if there exist a
normal complex variety \(Y\), a closed point \(y\in Y\), a morphism
\(f:Y\to \mathbb C^q\), and a foliation \(\mathcal G\) on \(Y\), such that:
\begin{enumerate}
\item \(\mathcal G\) is induced by \(f\);
\item the germs \((X,\mathcal F,x)\) and \((Y,\mathcal G,y)\) are analytically
isomorphic.
\end{enumerate}

We say that \(\mathcal F\) is \emph{meromorphically integrable at \(x\)} if the same
condition holds with \(f:Y\dashrightarrow \mathbb C^q\) a rational map.
\end{definition}

\begin{remark}
Thus analytically integrable at \(x\) means that, after replacing the analytic germ
\((X,\mathcal F,x)\) by an analytically isomorphic algebraic germ, the foliation admits
an algebraic first integral. 
\end{remark}

\begin{example}
The separatrices of the foliation in Example~\ref{ex2} through the origin are
\((x=0)\) and \((y=0)\), but the leaves through other points are of the form
\(y=Ce^{-1/x}\). Their Euclidean closures contain the origin, but they are not analytic
subvarieties in any neighbourhood of the origin. Hence \(\mathcal F\) does not admit a
holomorphic first integral at the origin, and in particular it is not analytically
integrable there.
\end{example}

\begin{example}
Let \(X=\mathbb C^2\), and let \(\mathcal F\) be the foliation generated by
\(
x\partial_x-\lambda y\partial_y,
\)
where \(\lambda\in\mathbb C\). If \(\lambda\in\mathbb Q_{>0}\), then \(\mathcal F\) is
both analytically \(p\)-closed and analytically integrable at the origin. Indeed, after
writing \(\lambda=a/b\) with \(a,b\in\mathbb Z_{>0}\), the function \(x^a y^b\) is a
holomorphic first integral.

On the other hand, if \(\lambda\notin\mathbb Q\), then by
Proposition~\ref{mcq08}, \(\mathcal F\) is not analytically \(p\)-closed at the origin.
Moreover, \(\mathcal F\) does not admit a holomorphic first integral at the origin, since
the leaves are of the form \(y=Cx^{-\lambda}\), which are not analytic subvarieties in any
neighbourhood of the origin. Hence \(\mathcal F\) is not analytically integrable at the
origin.
\end{example}

\subsection{Birational maps and hypersurface sections}

We prove several stability properties of analytic $p$-closedness that will be used
repeatedly.

\begin{lemma}
Let \(K\) be a field finitely generated over \(\mathbb Q\), let
\(
\pi:Y\longrightarrow X
\)
be a proper birational morphism between normal \(K\)-varieties, and let
\(\mathcal F\) be a foliation on \(X\) which is \(p\)-closed for almost
all primes. Then the strict transform
\(
\mathcal G:=\pi^{-1}\mathcal F
\)
is \(p\)-closed for almost all primes.
\end{lemma}

\begin{proof}
Choose a finitely generated integral
\(\mathbb Z\)-subalgebra \(R\subset K\), and spread out
\(X,Y,\pi,\mathcal F\), and \(\mathcal G\) over
\(S:=\operatorname{Spec}R\). After localising \(R\), we may assume that
the relevant fibres are normal and integral, that the induced
morphisms
\(
\pi_s:Y_s\longrightarrow X_s
\)
are birational, and that the reductions of \(\mathcal F\) are
\(p(s)\)-closed.

Choose a nonempty open subset \(V\subset X\) over which \(\pi\) is an
isomorphism. After shrinking \(S\), this open subset and the
isomorphism
\(
\pi^{-1}(V)\longrightarrow V
\)
spread out over \(S\).

Let \(s\in S\) be a closed point for which \(\mathcal F_s\) is
\(p(s)\)-closed. On the nonempty open subset
\(\pi_s^{-1}(V_s)\subset Y_s\), the strict transform
\(\mathcal G_s\) is identified with \(\mathcal F_s|_{V_s}\), and is
therefore \(p(s)\)-closed. Since \(p\)-closedness of a foliation on a
normal integral variety may be checked on any nonempty open subset,
Lemma~\ref{pcloseopenpositivechar} implies that
\(\mathcal G_s\) is \(p(s)\)-closed on all of \(Y_s\).

Thus \(\mathcal G\) is \(p\)-closed for almost all primes.
\end{proof}
\begin{corollary}\label{pclosebirational}
Let \((X,\mathcal F,x)\) be an analytically \(p\)-closed foliated germ,
and let \((X_0,\mathcal F_0,x_0)\) be a local \(p\)-closed model of it. Let
\(
\pi:Y\longrightarrow X_0
\)
be a proper birational morphism from a normal variety, and set
\(\mathcal G:=\pi^{-1}\mathcal F_0\). Then \(\mathcal G\) is
analytically \(p\)-closed at every point of \(\pi^{-1}(x_0)\).
\end{corollary}

\begin{proof}
By Lemma~\ref{pclosebirational}, the foliation
\(\mathcal G\) is \(p\)-closed for almost all primes. Hence, for every
\(y\in\pi^{-1}(x_0)\), the algebraic germ
\(
(Y,\mathcal G,y)
\)
is a local \(p\)-closed model of itself. It is therefore analytically
\(p\)-closed at \(y\).
\end{proof}
\begin{lemma}\label{quasietalepreservepclosed}
Let
\(
\pi:Y'\longrightarrow Y
\)
be a finite quasi-étale morphism between normal varieties. Let
\(\mathcal F\) be a foliation on \(Y\) which is \(p\)-closed for almost
all primes. Then the pull-back foliation
\(
\mathcal F':=\pi^{-1}\mathcal F
\)
is \(p\)-closed for almost all primes.
\end{lemma}

\begin{proof}
Choose a subfield \(K\subset\mathbb C\), finitely generated over
\(\mathbb Q\), over which \(Y\), \(Y'\), \(\pi\), \(\mathcal F\), and
\(\mathcal F'\) are defined. Choose a finitely generated integral
\(\mathbb Z\)-subalgebra \(R\subset K\) and spread out these data over
\(
S:=\operatorname{Spec}R.
\)
After replacing \(R\) by a localisation, we obtain normal integral
models
\(
\mathscr Y\longrightarrow S,
\mathscr Y'\longrightarrow S,
\)
a finite morphism
\(
\Pi:\mathscr Y'\longrightarrow\mathscr Y
\)
with generic fibre \(\pi\), and foliations
\(
\mathscr F\subset T_{\mathscr Y/S},
\mathscr F'\subset T_{\mathscr Y'/S},
\)
whose generic fibres are \(T_{\mathcal F}\) and \(T_{\mathcal F'}\).

Since \(\pi\) is quasi-étale, there is a nonempty open subset
\(V\subset Y\), whose complement has codimension at least two, such that
\(V\) is smooth and
\(
\pi^{-1}(V)\longrightarrow V
\)
is étale. Shrinking \(V\) further does not affect the argument, so we may
also assume that \(\mathcal F\) is regular on \(V\).

After localising \(R\) once more, \(V\) spreads out to an open subset
\(\mathscr V\subset\mathscr Y\) such that, for every closed point \(s\)
under consideration, the fibre \(\mathscr V_s\) is nonempty and smooth,
and
\(
\Pi_s^{-1}(\mathscr V_s)\longrightarrow\mathscr V_s
\)
is étale. We may moreover assume that \(\mathscr Y_s\) and
\(\mathscr Y'_s\) are normal and integral and that \(\mathscr F_s\) is
\(p(s)\)-closed for every closed point \(s\) in a nonempty open subset of
\(S\).

Fix such a closed point \(s\), and set
\(
V_s':=\Pi_s^{-1}(\mathscr V_s).
\)
On \(V_s'\), the saturated pull-back foliation agrees with the ordinary
étale pull-back:
\[
\mathscr F'_s|_{V_s'}
\simeq
\Pi_s^*(\mathscr F_s|_{\mathscr V_s}).
\]
Under the étale identifications of the tangent bundles and the tangent
sheaves of the foliations, the \(p\)-curvature map commutes with
pull-back. Consequently,
\[
\psi_{\mathscr F'_s}|_{V_s'}
=
\Pi_s^*
\bigl(
\psi_{\mathscr F_s}|_{\mathscr V_s}
\bigr)
=
0.
\]
Thus \(\mathscr F'_s\) is \(p(s)\)-closed on the nonempty open subset
\(V_s'\).

By Lemma~\ref{pcloseopenpositivechar}, \(p(s)\)-closedness on the normal
integral variety \(\mathscr Y'_s\) may be checked on any nonempty open
subset. Hence \(\mathscr F'_s\) is \(p(s)\)-closed on all of
\(\mathscr Y'_s\).

This holds for every closed point in a nonempty open subset of \(S\).
Therefore \(\mathcal F'\) is \(p\)-closed for almost all primes.
\end{proof}

\begin{definition}\label{generalhypersurfaceinlocalmodel}
Let \(X\) be a normal variety and let \(\mathcal F\) be a foliation on \(X\).
Suppose that \(\mathcal F\) is analytically \(p\)-closed at a closed point \(x\in X\),
and let \((Y,\mathcal G,y)\) be a local \(p\)-closed model of
\((X,\mathcal F,x)\).

A \emph{general hypersurface in the local \(p\)-closed model} is an analytic hypersurface
\(H\) in a sufficiently small analytic neighbourhood of \(x\) obtained as follows. Choose
an analytic isomorphism between neighbourhoods of \(x\in X^{\mathrm{an}}\) and
\(y\in Y^{\mathrm{an}}\) identifying \(\mathcal F\) and \(\mathcal G\). If
\(S\subset Y\) is a sufficiently general algebraic hypersurface through \(y\), then
\(H\) is the inverse image of \(S^{\mathrm{an}}\) under this analytic isomorphism.
\end{definition}

\begin{lemma}\label{hypersurfacesectionpclosed}
Let \(X\) be a normal variety, and let \(\mathcal F\) be a foliation on \(X\)
which is analytically \(p\)-closed at a closed point \(x\in X\). Let \(H\) be a general
hypersurface in a local \(p\)-closed model of \((X,\mathcal F)\) at \(x\). Then the
restricted foliation \(\mathcal F|_H\) is analytically \(p\)-closed at \(x\).
\end{lemma}

\begin{proof}
Choose a local \(p\)-closed model \((Y,\mathcal G,y)\) of \((X,\mathcal F,x)\) and an
analytic isomorphism of germs identifying \((X,\mathcal F,x)\) with
\((Y,\mathcal G,y)\), such that \(H\) corresponds to \(S^{\mathrm{an}}\), where
\(S\subset Y\) is a sufficiently general algebraic hypersurface through \(y\).

By the generality of \(S\), the restricted foliation \(\mathcal G|_S\) is defined. Since
\(\mathcal G\) is \(p\)-closed for almost all primes, it follows from
\cite[Lecture~3, 1.10]{MiyaokaPeternell1997} that \(\mathcal G|_S\) is also
\(p\)-closed for almost all primes. Therefore \((S,\mathcal G|_S,y)\) is a local
\(p\)-closed model of \((H,\mathcal F|_H,x)\). Hence \(\mathcal F|_H\) is analytically
\(p\)-closed at \(x\).
\end{proof}

\section{Holomorphic first integrals under canonical singularities}
\label{canonicalintegrabilitysection}

This section proves Theorem~\ref{theoremA}. We begin with the smooth surface
case, pass to canonical foliations on klt surface germs by a finite quasi-\'etale
cover, and then treat the higher-dimensional canonical cases by transverse surface
sections.

\subsection{The smooth surface case}

The following formal linearisation result is the main input in the proof of
Theorem~\ref{canonicalcases}\textup{(1)}.

\begin{proposition}[{\cite[Proposition II.1.3]{McQuillan2008}}]\label{mcq08}
Let $X$ be smooth and let $\mathcal F$ be a rank-one foliation with log canonical
singularities which is $p$-closed for almost all primes. For every
$z\in\operatorname{Sing}(\mathcal F)$, there are formal coordinates
$x_1,\ldots,x_s$, where $s\le\dim X$, and integers $\lambda_1,\ldots,\lambda_s$ such
that $\mathcal F$ is formally generated by
\[
\sum_{i=1}^s\lambda_i x_i\frac{\partial}{\partial x_i}.
\]
\end{proposition}

\begin{remark}
McQuillan states the result in \cite{McQuillan2008} under the assumption of
canonical singularities; the same proof applies to log canonical singularities.
\end{remark}

We first exclude saddle-nodes and then prove Theorem~\ref{canonicalcases}\textup{(1)}.

\begin{lemma}\label{pclosednotsaddle}
Let $X$ be a smooth surface, and let $\mathcal F$ be a
foliation on $X$ with log canonical singularities. If $\mathcal F$ is analytically
$p$-closed at a closed point $x$, then $x$ is not a saddle-node of $\mathcal F$.
\end{lemma}

\begin{proof}
Choose a local $p$-closed model $(Y,\mathcal G,y)$ of $(X,\mathcal F,x)$. Applying
Proposition~\ref{mcq08} at $y$, we obtain formal coordinates in which the foliation is
generated by a diagonal vector field with integral weights. Since a saturated rank-one
foliation on a smooth surface has singular locus of codimension at least two, its
singularity at $y$ is isolated. Consequently both coordinates occur with nonzero
weights: if one weight vanished, the formal singular locus would contain a divisor.
Thus the linear part has two nonzero eigenvalues, and the singularity cannot be a
saddle-node.
\end{proof}
\begin{proof}[Proof of Theorem~\ref{canonicalcases}\textup{(1)}]

We may assume that $x\in \operatorname{Sing}(\mathcal F)$, since if $\mathcal F$ is
regular at $x$, then the theorem follows directly from the Frobenius theorem.

Assume that $\mathcal F$ is analytically $p$-closed at $x$. Then there exist a normal
variety $Y$, a closed point $y\in Y$, and a foliation $\mathcal G$ on $Y$ such that
$(Y,\mathcal G,y)$ is analytically isomorphic to $(X,\mathcal F,x)$, and such that
$\mathcal G$ is $p$-closed for almost all primes $p$. By Proposition~\ref{mcq08}, there
exist formal coordinates $(\hat u,\hat v)$ at $x$ such that $\mathcal F$ is formally
generated by
\[
\delta=\eta \hat u\partial_{\hat u}+\mu \hat v\partial_{\hat v},
\]
where $\eta,\mu\in\mathbb Z$. By Lemma~\ref{pclosednotsaddle}, we have $\eta\mu\neq0$.

We first exclude the case $\eta\mu>0$. In this case the quotient
$\eta/\mu$ is a positive real number. Since the foliation is already formally
linearized by Proposition~\ref{mcq08}, Theorem~\ref{Poincare} shows that this
formal linearisation is convergent. Thus, in analytic coordinates, the foliation is
generated by
\[
\eta u\frac{\partial}{\partial u}
+
\mu v\frac{\partial}{\partial v}.
\]
Such a singularity is not canonical by
\cite[Fact I.i.19]{McQuillan2008}, a contradiction. Hence $\eta\mu<0$.

Replacing $\delta$ by $-\delta$ if necessary, we may assume that $\eta<0$ and $\mu>0$.
After dividing the two weights by their greatest common divisor, set
$a:=\mu/\gcd(\mu,-\eta)$ and $b:=-\eta/\gcd(\mu,-\eta)$. Then $a,b>0$ are
coprime, and in the formal coordinates $(\hat u,\hat v)$ the foliation is generated
by
\(
-b\hat u\partial_{\hat u}+a\hat v\partial_{\hat v}.
\)
Equivalently, it is defined by the formal $1$-form
$a\hat v\,d\hat u+b\hat u\,d\hat v$. Therefore
$\hat f:=\hat u^a\hat v^b$ is a formal first integral of $\mathcal F$, since
\[
d\hat f
=
\hat u^{a-1}\hat v^{b-1}
\big(a\hat v\,d\hat u+b\hat u\,d\hat v\big).
\]

By \cite[Théorème A]{MatteiMoussu1980}, after composing $\hat f$ with a
one-variable formal change of coordinate, the formal first integral becomes
convergent. Thus there exists
$\varphi\in \widehat{\mathcal O}_{\mathbb C,0}$ with $\varphi(0)=0$ and
$\varphi'(0)\neq0$ such that $F:=\varphi(\hat f)$ is a holomorphic first integral of
$\mathcal F$.

Since $\varphi(0)=0$ and $\varphi'(0)\neq0$, we can write
$\varphi(t)=t\psi(t)$ with $\psi(0)\neq0$. Hence
\[
F=\varphi(\hat f)
 =\hat f\psi(\hat f)
 =\hat u^a\hat v^b\psi(\hat u^a\hat v^b).
\]
Thus
\(
F=\hat u^a\hat v^b\widehat\varepsilon
\)
for some formal unit $\widehat\varepsilon\in\widehat{\mathcal O}_{X,x}^{\times}$.

The formal curves $\{\hat u=0\}$ and $\{\hat v=0\}$ are the two formal separatrices of
$\mathcal F$. Since $\mathcal F$ has canonical singularities, these formal separatrices
are convergent by \cite[Appendix~II]{MatteiMoussu1980}. Moreover, they are formally
smooth and transverse. Hence the corresponding analytic separatrices are smooth and
transverse. After an analytic change of coordinates, we may therefore assume that they
are given by $u=0$ and $v=0$.

Since $X$ is smooth at $x$, the local ring $\mathcal O_{X,x}$ is a UFD. The holomorphic
first integral $F$ has, formally, the factors corresponding to the two separatrices with
multiplicities $a$ and $b$. Hence, by unique factorization in $\mathcal O_{X,x}$, we can
write
\(
F=u^a v^b\varepsilon
\)
for some holomorphic unit $\varepsilon\in\mathcal O_{X,x}^{\times}$.

After shrinking the analytic neighbourhood of $x$, choose a holomorphic $b$-th root
$\varepsilon^{1/b}$. Set $z=u$ and $w=v\varepsilon^{1/b}$. Then
\(
F=z^a w^b.
\)
Thus \(F\) is a holomorphic first integral near \(x\). Therefore $\mathcal F$ is
analytically integrable at $x$ in the sense of
Definition~\ref{definitionanalyticintegral}.

Conversely, suppose that $\mathcal F$ is analytically integrable at $x$. By definition,
there exist a normal algebraic variety $Y$, a closed point $y\in Y$, and a morphism
$f:Y\to \mathbb A^1$ such that $(Y,\mathcal G,y)$ is analytically isomorphic to
$(X,\mathcal F,x)$, where $\mathcal G$ is the foliation induced by $f$. It is enough to
show that $\mathcal G$ is $p$-closed for almost all primes $p$.

After spreading out, we may assume that \(Y\), \(f\), and \(\mathcal G\) are defined over
a finitely generated \(\mathbb Z\)-algebra. For all closed fibres outside a finite set of
primes, the reduction of \(\mathcal G\) is still the foliation induced by the reduction
of \(f\). On the open set where \(df\neq0\), this foliation is given by
\(
T_{\mathcal G}=\ker(df).
\)
Hence, for any local vector field \(v\in T_{\mathcal G}\), we have \(v(f)=0\). In characteristic
\(p\), this implies
\[
v^p(f)=v^{p-1}(v(f))=0.
\]
Therefore \(v^p\in\ker(df)=T_{\mathcal G}\) on this open set. Since \(T_{\mathcal G}\) is
saturated, the inclusion \(v^p\in T_{\mathcal G}\) extends across the complement. Thus the
reduction of \(\mathcal G\) is \(p\)-closed for almost all primes \(p\). Hence
\(\mathcal F\) is analytically \(p\)-closed at \(x\).

Moreover, the preceding argument gives holomorphic coordinates $z,w$ at $x$ for
which $F=z^aw^b$. Since
\[
dF=z^{a-1}w^{b-1}(aw\,dz+bz\,dw),
\]
the foliation is generated by $aw\partial_w-bz\partial_z$. Hence it is
holomorphically linearizable.

\end{proof}

\subsection{Canonical foliations on klt surface germs}

The smooth surface result also gives the canonical klt surface case after passing to
a finite quasi-\'etale smooth cover.

\begin{proof}[Proof of
Theorem~\ref{canonicalcases}\textup{(2)}]

Choose a local \(p\)-closed model
\(
(X_0,\mathcal F_0,x_0)
\)
of \((X,\mathcal F,x)\). Thus \(\mathcal F_0\) is \(p\)-closed for
almost all primes, and there is an analytic isomorphism of foliated germs
\(
(X_0,\mathcal F_0,x_0)
\simeq
(X,\mathcal F,x).
\)
In particular, \(X_0\) is klt at \(x_0\), and \(\mathcal F_0\) has
canonical singularities there.

A klt surface singularity is a quotient singularity; see
\cite[Proposition~4.18]{KollarMori1998}. We use the étale-local algebraic
form of this statement. After replacing \((X_0,x_0)\) by an algebraic
étale neighbourhood
\(
\tau:(X_1,x_1)\longrightarrow(X_0,x_0),
\)
there exists a finite Galois quasi-étale morphism
\(
q:(S,s)\longrightarrow(X_1,x_1)
\)
such that \(S\) is smooth and
\(
q^{-1}(x_1)=\{s\}.
\)
In particular, the Galois group fixes \(s\).

Set
\(
\mathcal F_1:=\tau^{-1}\mathcal F_0,
\mathcal H:=q^{-1}\mathcal F_1.
\) We first check the arithmetic hypothesis on \(\mathcal H\). The foliation
\(\mathcal F_1\) is \(p\)-closed for almost all primes. Indeed, after
spreading out \(\tau\) together with \(\mathcal F_0\), the reductions of
\(\tau\) are étale over a nonempty open subset of the arithmetic base.
On every such closed fibre, the \(p\)-curvature of the pulled-back
foliation is the étale pull-back of the \(p\)-curvature of
\(\mathcal F_0\), and hence vanishes.

Lemma~\ref{quasietalepreservepclosed} now implies that
\(\mathcal H\) is \(p\)-closed for almost all primes. Consequently, the
algebraic germ
\(
(S,\mathcal H,s)
\)
is itself a local \(p\)-closed model. By
Definition~\ref{definitionpclosed}, \(\mathcal H\) is therefore
analytically \(p\)-closed at \(s\).

Étale pull-back preserves canonical singularities, and canonical
singularities of foliations are also preserved by finite quasi-étale
pull-back; see
\cite[Corollary~III.i.5]{McQuillanPanazzolo2013}. Hence
\(\mathcal H\) has canonical singularities at \(s\). Since \(S\) is
smooth, Theorem~\ref{canonicalcases}\textup{(1)} gives a nonconstant
holomorphic first integral
\(
f:(S,s)\longrightarrow(\mathbb C,0)
\)
of \(\mathcal H\). After subtracting a constant, we may assume that
\(f(s)=0\).

Let \(\Gamma\) be the Galois group of \(q\), and define the holomorphic
germ
\[
\widetilde f
:=
\prod_{\gamma\in\Gamma} f\circ\gamma.
\]
Every deck transformation preserves \(\mathcal H\), and hence every
factor \(f\circ\gamma\) is a first integral of \(\mathcal H\). It follows
that \(\widetilde f\) is also a first integral.

The germ \(\widetilde f\) is nonzero because
\(\mathcal O_{S,s}\) is an integral domain and none of the factors is
identically zero. Since every element of \(\Gamma\) fixes \(s\), all the
factors vanish at \(s\). Thus \(\widetilde f(s)=0\), and
\(\widetilde f\) is nonconstant.

The group \(\Gamma\) permutes the factors in the above product, so
\(\widetilde f\) is \(\Gamma\)-invariant. Since \(q\) is a finite Galois
quotient, the identity
\[
\mathcal O_{X_1,x_1}
=
\bigl(q_*\mathcal O_{S,s}\bigr)^\Gamma
\]
shows that \(\widetilde f\) descends to a holomorphic germ
\(
h_1:(X_1,x_1)\longrightarrow(\mathbb C,0)
\)
satisfying \(q^*h_1=\widetilde f\).

On the dense open subset where \(q\) is étale and \(dh_1\neq0\), the
function \(h_1\) is constant along \(\mathcal F_1\). Indeed, its pull-back
\(\widetilde f\) is constant along
\(\mathcal H=q^{-1}\mathcal F_1\). Therefore
\(
T_{\mathcal F_1}\subseteq\ker(dh_1)
\)
on this open subset. Both sheaves have rank one, and hence they are
equal there. Since both define saturated rank-one subsheaves of
\(T_{X_1}\), the equality extends to the whole germ. Thus \(h_1\) is a
holomorphic first integral of \(\mathcal F_1\) at \(x_1\).

Finally, the analytification of the étale morphism \(\tau\) is a local
biholomorphism at \(x_1\). After shrinking analytic neighbourhoods,
\(h_1\) therefore transports to a holomorphic first integral of
\(\mathcal F_0\) at \(x_0\). Transporting once more through the analytic
identification
\(
(X_0,\mathcal F_0,x_0)
\simeq
(X,\mathcal F,x)
\)
gives a holomorphic first integral of \(\mathcal F\) at \(x\).
\end{proof}

\subsection{Higher-dimensional canonical singularities}

We now prove Theorem~\ref{canonicalcases}\textup{(3)} and record its immediate
higher-dimensional consequences.

\begin{lemma}[{\cite[Corollaire C]{MatteiMoussu1980}}]\label{MM80}
Let $\mathcal F$ be a corank one foliation on the analytic germ
$(0\in\mathbb C^n)$, and let
$h:(0\in\mathbb C^m)\to(0\in\mathbb C^n)$ be a holomorphic map with $m<n$.
Suppose that $h$ is transverse to $\mathcal F$, that is,
\[
\operatorname{Sing}(h^*\omega)=h^{-1}(\operatorname{Sing}(\omega)),
\]
and $$\operatorname{codim} (\operatorname{Sing}(h^*\omega))=\min(m, \operatorname{codim}(\operatorname{Sing}(\omega))),$$
where $\omega$ is a local holomorphic $1$-form defining $\mathcal F$, and Then
$\mathcal F$ admits a holomorphic first integral if and only if the pull-back
foliation $h^{-1}\mathcal F$ does.
\end{lemma}

\begin{proof}[Proof of Theorem~\ref{canonicalcases}\textup{(3)}]
If every component of $\operatorname{Sing}(\mathcal F)$ through $x$ has codimension
at least three, the result follows from \cite[Corollary 1]{CerveauLinsNeto2008}.
We may therefore assume that $x$ is a general point of a codimension two component.

Choose a general smooth surface $S\subset X$ through $x$, transverse to that
component. By \cite[Proposition 3.5]{LorayPereiraTouzet2018}, the restriction
$\mathcal F|_S$ has canonical singularities. Applying
Lemma~\ref{hypersurfacesectionpclosed} successively shows that it is analytically
$p$-closed at $x$. Theorem~\ref{canonicalcases}\textup{(1)} gives a holomorphic first
integral on $S$, and Lemma~\ref{MM80} lifts it to $X$.
\end{proof}

\begin{corollary}\label{canonicaloutsidecodimthree}
Let $X$ be smooth and let $\mathcal F$ be a corank one foliation with canonical
singularities which is $p$-closed for almost all primes. Then there is a
closed subset $Z\subset X$ of codimension at least three such that $\mathcal F$ admits
a local holomorphic first integral at every point of $X\setminus Z$.
\end{corollary}

\begin{proof}
Let $S_1,\ldots,S_r$ be the codimension two irreducible components of
$\operatorname{Sing}(\mathcal F)$. For each $i$, Theorem~\ref{canonicalcases}\textup{(3)}
applies on a dense open subset $S_i^\circ\subset S_i$. Set
\[
Z:=\bigcup_i(S_i\setminus S_i^\circ)
\]
and enlarge it, if necessary, by a closed subset of codimension at least three. Outside
$Z$, the conclusion follows from Frobenius at regular points, from
Theorem~\ref{canonicalcases}\textup{(3)} on the general codimension two strata, and from
\cite[Corollary 1]{CerveauLinsNeto2008} at points lying only on components of
codimension at least three.
\end{proof}

\begin{corollary}\label{threefoldoutsidefinite}
Let \(X\) be a smooth threefold and let \(\mathcal F\) be a corank one foliation
with canonical singularities. Suppose that \(\mathcal F\) is 
\(p\)-closed for almost all primes. Then there exists a finite subset
\(
T\subset\operatorname{Sing}(\mathcal F)
\)
such that \(\mathcal F\) admits a local holomorphic first integral at every point
of \(X\setminus T\).
\end{corollary}

\begin{proof}
Apply Corollary~\ref{canonicaloutsidecodimthree}. On a threefold, a closed subset
of codimension at least three is finite.
\end{proof}

\section{Finiteness of holonomy}
\label{holonomysection}

This section proves Theorem~\ref{theoremA}. We first recall the notions of
holonomy and local holonomy, and then construct flat connections which record the
finite transverse jets of the holonomy representation.

\subsection{Finite holonomy along invariant divisors}

\begin{definition}[Holonomy]\label{definitionholonomy}
Let \(X\) be a smooth variety, let \(\mathcal F\) be a foliation of corank
\(q\), and let \(L\) be a leaf. Choose points \(p_0,p_1\in L\) and germs
\((\Sigma_0,p_0)\), \((\Sigma_1,p_1)\) of smooth \(q\)-dimensional
submanifolds transverse to \(\mathcal F\). Analytic continuation through
foliation flow boxes along a path \(\gamma\) from \(p_0\) to \(p_1\) gives a
germ of biholomorphism
\(h_\gamma:(\Sigma_0,p_0)\to(\Sigma_1,p_1)\), depending only on the
homotopy class of \(\gamma\) relative to its endpoints.

For loops based at \(p_0\), these germs define the holonomy representation
\(\rho_{L,p_0,\Sigma_0}:\pi_1(L,p_0)\to
\operatorname{Diff}(\Sigma_0,p_0)\). Its image is the \emph{holonomy group}
of \(L\). If \(V\subset L\) is a connected open subset containing \(p_0\),
the image of \(\pi_1(V,p_0)\) under the same representation is called the
holonomy group of \(V\).
\end{definition}

\begin{definition}[Local holonomy]\label{definitionlocalholonomy}
Let \(L\) be a leaf and let \(p\in\overline L\setminus L\). For a sufficiently
small neighbourhood \(B\) of \(p\), a connected component of \(L\cap B\)
whose closure contains \(p\) is called a \emph{local branch of \(L\) at
\(p\)}. The holonomy group of such a branch, computed from a point and a
transversal inside the branch, is its \emph{local holonomy group}. After
choosing a path to a fixed base point of \(L\), it is identified, up to
conjugacy, with a subgroup of the holonomy group of \(L\).
\end{definition}

\begin{remark}\label{remarklocalholonomy}
Holonomy groups depend on the chosen base point and transversal only by
conjugacy. Thus finiteness, cyclicity, and the equality of a local holonomy
subgroup with the full holonomy group do not depend on these choices.

\end{remark}

\begin{lemma}\label{gmgaarithmetic}
Let \(B\) be a smooth connected complex variety, and let
\(\beta\in H^0(B,\Omega_B^1)\) be a closed algebraic \(1\)-form.
Consider
\[
\pi:P_\beta:=B\times\mathbb A^1_s\longrightarrow B
\]
as a principal \(\mathbb G_a\)-bundle via translations in the fibre
coordinate \(s\), and set
\[
H_\beta:=\ker\bigl(ds+\pi^*\beta\bigr)
\subset T_{P_\beta}.
\]
Then \(H_\beta\) is a flat principal \(\mathbb G_a\)-connection.

Suppose that, after spreading out, its reductions have vanishing
\(p\)-curvature for almost all primes. Then there exists a regular
function \(c\in H^0(B,\mathcal O_B)\) such that
\(
dc=-\beta.
\)
Equivalently, the graph
\[
B\longrightarrow P_\beta,
\qquad
b\longmapsto (b,c(b)),
\]
is a horizontal section.
\end{lemma}

\begin{proof}
For a local vector field \(v\) on \(B\), the unique horizontal lift is
\[
\widetilde v
=
v-\beta(v)\frac{\partial}{\partial s}.
\]
Thus \(H_\beta\) is complementary to the vertical tangent bundle.
Translations in \(s\) preserve \(ds+\pi^*\beta\), so the distribution
is invariant under the \(\mathbb G_a\)-action and hence defines a
principal connection. Moreover,
\[
[\widetilde v,\widetilde w]
=
\widetilde{[v,w]}
-
d\beta(v,w)\frac{\partial}{\partial s}.
\]
Since \(d\beta=0\), the connection is flat.

Consider the rank-two connection
\[
\mathcal E_\beta
=
\mathcal O_Be_0\oplus\mathcal O_Be_1,
\qquad
\nabla e_0=0,
\qquad
\nabla e_1=\beta\otimes e_0.
\]
It fits into an exact sequence
\[
0\longrightarrow
(\mathcal O_B,d)
\longrightarrow
(\mathcal E_\beta,\nabla)
\longrightarrow
(\mathcal O_B,d)
\longrightarrow0.
\]
The \(p\)-curvature of this connection is the image of the
\(p\)-curvature of \(H_\beta\) under the standard faithful
representation
\[
\mathbb G_a\longrightarrow\operatorname{GL}_2,
\qquad
a\longmapsto
\begin{pmatrix}
1&a\\
0&1
\end{pmatrix}.
\]
It therefore vanishes after reduction modulo almost all primes.

By Theorem~\ref{andreextension},
\((\mathcal E_\beta,\nabla)\) becomes trivial after a finite étale
cover. Passing to a connected finite Galois cover
\(
\tau:B'\longrightarrow B
\)
dominating such a cover, we may choose a global horizontal frame
\(u_0,u_1\) of \(\tau^*\mathcal E_\beta\).

Let
\(
q:\tau^*\mathcal E_\beta\longrightarrow\mathcal O_{B'}
\)
be the pull-back of the quotient map in the preceding exact sequence.
Since \(q\) is a morphism of flat connections, the functions
\(q(u_0)\) and \(q(u_1)\) are horizontal for the trivial connection
\((\mathcal O_{B'},d)\), and are therefore constant. Since \(q\) is
surjective, they are not both zero. After taking a constant linear
combination of \(u_0\) and \(u_1\), we obtain a horizontal section
\(u\) such that
\(
q(u)=1.
\)

We continue to denote the pull-backs of \(e_0,e_1\) by the same
symbols. Since
\(
q(e_1)=1, 
\ker(q)=\mathcal O_{B'}e_0,
\)
there is a unique regular function
\(f\in H^0(B',\mathcal O_{B'})\) such that
\(
u=e_1+fe_0.
\)
The horizontality of \(u\) gives
\[
0
=
\nabla u
=
\bigl(df+\tau^*\beta\bigr)\otimes e_0,
\]
and hence
\(
df=-\tau^*\beta.
\) Let \(G\) be the Galois group of \(\tau\), and set
\[
\overline f
:=
\frac{1}{|G|}
\sum_{\sigma\in G}\sigma^*f.
\]
Then \(\overline f\) is \(G\)-invariant and still satisfies
\(
d\overline f=-\tau^*\beta.
\)
It therefore descends to a regular function
\(c\in H^0(B,\mathcal O_B)\) satisfying \(dc=-\beta\).
\end{proof}

The following proof is inspired by the formal-neighbourhood strategy of
\cite[Proposition~2.5]{EkedahlShepherdBarronTaylor1999}.  The higher-order
lifting step is made explicit below: at each order, the obstruction is encoded by a flat principal
\(\mathbb G_a\)-connection, and Lemma~\ref{gmgaarithmetic} forces it to vanish.

\begin{proof}[Proof of Theorem~\ref{theoremA}]
Put \(n:=\dim X\). Since \(D\) is \(\mathcal F\)-invariant, on
\(D^\circ\) one has
\(T_{\mathcal F}|_{D^\circ}\subseteq T_{D^\circ}\). Both bundles have
rank \(n-1\), so they are equal. Hence each connected component of
\(D^\circ\) is a leaf. Fix one such component \(L\).

\smallskip
\noindent\textbf{Step 1: Choose an étale chart and a local frame near $L$}

All the algebraic data descend to a subfield \(K\subset\mathbb C\), finitely
generated over \(\mathbb Q\). After a finite extension of \(K\) and
shrinking around the generic point of \(D\), choose an affine open subset
\(W\subset X\) and regular functions \(x_1,\ldots,x_{n-1},y\) such that
\(
(x_1,\ldots,x_{n-1},y):W\longrightarrow\mathbb A_K^n\) is étale, 
and \( U:=D\cap W=(y=0)\subset L.\) After shrinking \(W\) once more, $T_\FF$ has
the unique frame
\[
\xi_i=\partial_{x_i}+A_i(x,y)\partial_y,
\qquad A_i(x,0)=0,
\qquad i=1,\ldots,n-1,
\]
normalised by \(\xi_i(x_j)=\delta_{ij}\). The fields \(\xi_i\) commute.
Indeed, involutivity gives
\([\xi_i,\xi_j]=\sum_k c_{ij}^k\xi_k\), and applying this equality to every
\(x_\ell\) gives \(c_{ij}^\ell=0\).

Choose \(q\in U\), and let
\[
\Sigma
=
\{x_1=x_1(q),\ldots,x_{n-1}=x_{n-1}(q)\}
\]
be the coordinate transversal, with transverse coordinate
\(t:=y|_\Sigma\). Let
\(\rho:\pi_1(L,q)\to\operatorname{Diff}(\mathbb C,0)\) be the holonomy
representation. The inclusion \(i:U(\mathbb C)\hookrightarrow L(\mathbb C)\)
induces a surjection on fundamental groups. Indeed, since \(L\) is smooth and \(L\setminus U\) is a proper complex
algebraic subset, every loop in \(L\) can be perturbed, relative to its
base point, to avoid \(L\setminus U\). Hence
\(i_*:\pi_1(U,q)\twoheadrightarrow\pi_1(L,q)\). It
is therefore enough to prove that \(\rho\circ i_*\) has finite image.

Let \(\mathcal I\subset\mathcal O_W\) be the ideal of \(U\), generated by
\(y\). For \(m\geq0\), write
\(U^{(m)}=(U,\mathcal O_W/\mathcal I^{m+1})\) for the \(m\)-th infinitesimal
neighbourhood of \(U\) in \(W\).

\smallskip
\noindent\textbf{Step 2: Kill the first-order holonomy.}

The conormal line bundle
\(
\mathcal N^*:=\mathcal I/\mathcal I^2
\)
carries the Bott connection. Since
\(T_{\mathcal F}|_U=T_U\), this is a full flat connection on \(U\).
For a local section \(\overline f\) of \(\mathcal N^*\), represented by
\(f\in\mathcal I\), one has
\[
\nabla^{\mathrm B}_{\partial_{x_i}}(\overline f)
=
\overline{\xi_i(f)}.
\]

We claim that its reductions have vanishing \(p\)-curvature. Indeed, \(p\)-closedness of $\FF$ gives
\(
\xi_i^{[p]}
=
\sum_j c_{ij}\xi_j.
\) after modulo $p$.
Applying this identity to \(x_k\) gives
\[
c_{ik}
=
\xi_i^{[p]}(x_k)
=
\xi_i^p(x_k)
=
0,
\]
because \(\xi_i(x_k)=\delta_{ik}\). Hence
\(\xi_i^{[p]}=0\). Moreover,
\(\partial_{x_i}^{[p]}=0\), since the \(x_i\) are étale coordinates.
It follows that
\[
\psi_{\nabla^{\mathrm B}}(\partial_{x_i})(\overline f)
=
\bigl(\nabla^{\mathrm B}_{\partial_{x_i}}\bigr)^p(\overline f)
-
\nabla^{\mathrm B}_{\partial_{x_i}^{[p]}}(\overline f)
=
\overline{\xi_i^p(f)}
=
0.
\]
Thus the reductions of
\((\mathcal N^*,\nabla^{\mathrm B})\) have vanishing \(p\)-curvature for
almost all primes.

By Theorem~\ref{rankoneGK}, the Bott connection
\((\mathcal N^*,\nabla^{\mathrm B})\) is isotrivial. Hence there exists
a connected finite étale cover
\(
\nu:U'\longrightarrow U
\)
such that \(\nu^*\mathcal N^*\) admits a nowhere-vanishing algebraic
horizontal section. Since \(U\) is affine and \(\nu\) is finite, \(U'\) is affine.

Since finite étale covers lift uniquely across nilpotent closed
immersions, \(\nu\) extends uniquely, compatibly in \(m\), to finite
étale covers
\(
\nu^{(m)}:U'^{(m)}\longrightarrow U^{(m)}.
\)
Let
\[
\mathcal I'_m
:=
\ker\bigl(
\mathcal O_{U'^{(m)}}\longrightarrow\mathcal O_{U'}
\bigr),
\]
and suppress the subscript \(m\) when the order is clear. Étale base
change gives a canonical identification
\[
\nu^*(\mathcal I/\mathcal I^2)
\simeq
\mathcal I'/\mathcal I'^2.
\]
Under this identification, denote the chosen horizontal generator by
\(z_1\).

The vector fields \(\xi_i\) preserve \(\mathcal I\), and hence induce
derivations on every \(U^{(m)}\). They lift uniquely through the étale
maps \(\nu^{(m)}\); we continue to denote the lifted vector fields by
\(\xi_i\). We also write \(x_i\) for the pull-back of \(x_i\) to
\(U'\), and set
\(
V_i:=\xi_i|_{U'}.
\)
Then
\(
V_i(x_j)=\delta_{ij},
\)
so \(V_1,\ldots,V_{n-1}\) form the frame of \(T_{U'}\) dual to
\(dx_1,\ldots,dx_{n-1}\).

\smallskip
\noindent\textbf{Step 3: Lift the horizontal coordinate order by order.}

We construct inductively, for every \(m\geq1\), compatible sections
\(z_m\in
H^0\bigl(U'^{(m)},\mathcal I'/\mathcal I'^{m+1}\bigr)
\)
whose image modulo \(\mathcal I'^2\) is \(z_1\), and which satisfy
\[
\xi_i(z_m)=0
\pmod{\mathcal I'^{m+1}}
\]
for every $i$.
The case \(m=1\) is given by the horizontal generator \(z_1\)
constructed in Step~2.

Suppose that \(m\geq2\) and that \(z_{m-1}\) has already been
constructed. On \(U'^{(m)}\) there is an exact sequence
\[
0\longrightarrow
\mathcal I'^m/\mathcal I'^{m+1}
\longrightarrow
\mathcal I'/\mathcal I'^{m+1}
\longrightarrow
\mathcal I'/\mathcal I'^m
\longrightarrow0.
\]
Since \(U'^{(m)}\) is affine, the induced map on global sections is
surjective. We may therefore choose a section
\(
z\in
H^0\bigl(U'^{(m)},\mathcal I'/\mathcal I'^{m+1}\bigr)
\)
whose image modulo \(\mathcal I'^m\) is \(z_{m-1}\).

This lift need not be horizontal. Since \(z_{m-1}\) is horizontal,
\(\xi_i(z)\) vanishes modulo \(\mathcal I'^m\), and hence determines a
section of \(\mathcal I'^m/\mathcal I'^{m+1}\).

Since \(U\subset W\) is a Cartier divisor, multiplication gives an
isomorphism
\[
\bigl(\mathcal I'/\mathcal I'^2\bigr)^{\otimes m}
\simeq
\mathcal I'^m/\mathcal I'^{m+1}.
\]
Via this isomorphism, the Bott connection induces a connection on
\(\mathcal I'^m/\mathcal I'^{m+1}\). The section \(z_1^m\) is a
nowhere-vanishing horizontal generator of this line bundle; in
particular,
\[
\xi_i(z_1^m)=0
\pmod{\mathcal I'^{m+1}}.
\]
There are therefore unique functions \(b_i\in\mathcal O(U')\) such that
\[
\xi_i(z)
=
b_i z_1^m
\pmod{\mathcal I'^{m+1}}.
\tag{\(*_m\)}
\]

The functions \(b_i\) measure the failure of \(z\) to be horizontal.
Applying \([\xi_i,\xi_j]=0\) to \(z\), and using the horizontality of
\(z_1^m\), gives
\[
0
=
[\xi_i,\xi_j](z)
=
\bigl(V_i(b_j)-V_j(b_i)\bigr)z_1^m
\pmod{\mathcal I'^{m+1}}.
\]
Since \(z_1^m\) is a generator,
\(V_i(b_j)=V_j(b_i)\). The vector fields \(V_i\) commute and are dual
to \(dx_1,\ldots,dx_{n-1}\); hence the algebraic one-form
\(
\beta_m:=\sum_i b_i\,dx_i
\)
is closed.

Replacing \(z\) by \(z+c z_1^m\), with \(c\in\mathcal O(U')\), replaces
the coefficients \(b_i\) by \(b_i+V_i(c)\), and therefore replaces
\(\beta_m\) by \(\beta_m+dc\). Thus it remains to prove that
\(\beta_m\) is exact.

Fix \(m\), spread out the finite-order data above, and work on a 
closed fibre of characteristic \(p\). By Step~2 and étaleness,
\(
\xi_i^{[p]}=0\) and
\(
V_i^{[p]}=0.
\)
Repeatedly applying \(\xi_i\) to \((\ast_m)\), and using the
horizontality of \(z_1^m\), yields
\[
0
=
\xi_i^p(z)
=
V_i^{p-1}(b_i)z_1^m
\pmod{\mathcal I'^{m+1}}.
\]
Since \(z_1^m\) is a generator, it follows that
\[
V_i^{p-1}(b_i)=0.
\tag{\(**_m\)}
\]

Consider the trivial principal \(\mathbb G_a\)-bundle
\(
P_m:=U'\times\mathbb A^1_s\longrightarrow U'
\)
with horizontal distribution
\(
H_m:=\ker(ds+\beta_m).
\)
Since \(d\beta_m=0\), Lemma~\ref{gmgaarithmetic} shows that \(H_m\)
is a flat principal connection. The horizontal lift of \(V_i\) is
\(
\widetilde V_i
=
V_i-b_i\frac{\partial}{\partial s}.
\)
On the same characteristic-\(p\) fibre, \(b_i\) is independent of
\(s\), \([V_i,\partial_s]=0\), and
\(\partial_s^{[p]}=0\). So under the natural identification
\(
\operatorname{ad}(P_m)
\simeq
\mathcal O_{U'}\frac{\partial}{\partial s},
\)
one has
\[
\psi_{H_m}(V_i)
=
-V_i^{p-1}(b_i)\frac{\partial}{\partial s}=0
\]
 by \((\ast\ast_m)\). Since the \(V_i\) form a frame of
\(T_{U'}\), the reductions of \(H_m\) have vanishing \(p\)-curvature
for almost all primes.

Lemma~\ref{gmgaarithmetic} therefore gives a regular function
\(c_m\in\mathcal O(U')\) satisfying \(dc_m=-\beta_m\). Set
\(
z_m:=z+c_mz_1^m.
\)
Then \(z_m\) still lifts \(z_{m-1}\), and
\[
\xi_i(z_m)
=
\bigl(b_i+V_i(c_m)\bigr)z_1^m
=
0
\pmod{\mathcal I'^{m+1}}.
\]
Thus \(z_m\) is the required horizontal lift, completing the induction.
\smallskip
\smallskip

\smallskip
\noindent\textbf{Step 4: Conclude finite holonomy.}

Choose \(q'\in U'\) over \(q\), and set
\[
\Pi'
:=
\nu_*\pi_1(U',q')
\subseteq
\pi_1(U,q).
\]
Since \(\nu:U'\to U\) is a connected finite étale cover, the subgroup
\(\Pi'\) has finite index in \(\pi_1(U,q)\). We prove that every element
of \(\Pi'\) has trivial holonomy.

Let \(\gamma\) be a loop in \(U'\) based at \(q'\), and let
\(\bar\gamma:=\nu\circ\gamma\) be the corresponding loop in \(U\).
Denote by \(h_{\bar\gamma}\) its holonomy germ on the fixed transversal
\((\Sigma,q)\). We take
\[
\Sigma
=
\{x_1=x_1(q),\ldots,x_{n-1}=x_{n-1}(q)\}
\]
and write \(t:=y|_\Sigma\) for its transverse coordinate.

For \(m\geq1\), let
\[
\Sigma^{(m)}
:=
\operatorname{Spec}\mathbb C[t]/(t^{m+1})
\]
be the \(m\)-th infinitesimal neighbourhood of \(q\) in \(\Sigma\).
Restricting \(\nu^{(m)}:U'^{(m)}\to U^{(m)}\) to
\(\Sigma^{(m)}\) gives a finite étale cover of \(\Sigma^{(m)}\).
Since \(\Sigma^{(m)}\) has no nontrivial connected finite étale covers,
the component containing \(q'\) is isomorphic to \(\Sigma^{(m)}\).
We use this isomorphism to identify the lifted infinitesimal
transversal at \(q'\) with \(\Sigma^{(m)}\).

Let
\(
\theta_m(t)
\in
t\mathbb C[t]/(t^{m+1})
\)
be the restriction of \(z_m\) to this lifted infinitesimal
transversal. Since \(z_m\) reduces to the nowhere-vanishing section
\(z_1\) modulo \(\mathcal I'^2\), one has
\[
\theta_m(t)
=
\lambda t+O(t^2)
\pmod{t^{m+1}}
\]
for some \(\lambda\neq0\). Hence \(\theta_m\) is invertible for
composition modulo \(t^{m+1}\).

Subdivide \(\bar\gamma\) into finitely many segments contained in
regular flow boxes. After shrinking the flow boxes, the cover
\(U'\to U\) is trivial over their intersections with \(U\), and the
lift \(\gamma\) selects compatible sheets. By uniqueness of finite
étale lifting across nilpotent thickenings, the same holds for
\(U'^{(m)}\to U^{(m)}\).

In flow-box coordinates
\((w_1,\ldots,w_{n-1},t)\), write
\[
z_m
=
\sum_{r=1}^m a_r(w)t^r
\pmod{t^{m+1}}.
\]
Since the vector fields \(\xi_i\) span the leafwise directions and
\[
\xi_i(z_m)=0
\pmod{\mathcal I'^{m+1}},
\]
each coefficient \(a_r\) is independent of
\(w_1,\ldots,w_{n-1}\). Thus the local holonomy in every flow box
preserves the restriction of \(z_m\) to the transversals. Composing
these local identities along \(\bar\gamma\), and using that
\(\gamma\) is a loop based at \(q'\), gives
\[
\theta_m\bigl(h_{\bar\gamma}(t)\bigr)
=
\theta_m(t)
\pmod{t^{m+1}}.
\]
Since \(\theta_m\) is invertible for composition, it follows that
\[
h_{\bar\gamma}(t)
=
t
\pmod{t^{m+1}}.
\]

This holds for every \(m\geq1\). Hence all Taylor coefficients of the
holomorphic germ \(h_{\bar\gamma}\) agree with those of the identity,
and therefore
\(
h_{\bar\gamma}
=
\operatorname{id}.
\)

Thus every element of \(\Pi'\) has trivial holonomy, so
\(
\Pi'
\subseteq
\ker(\rho\circ i_*).
\)
Since \(\Pi'\) has finite index in \(\pi_1(U,q)\), the kernel of
\(\rho\circ i_*\) also has finite index. Hence
\(\operatorname{Im}(\rho\circ i_*)\) is finite. Finally,
\(i_*:\pi_1(U,q)\twoheadrightarrow\pi_1(L,q)\) is surjective, and
therefore
\(
\operatorname{Im}(\rho)
=
\operatorname{Im}(\rho\circ i_*)
\)
is finite.
\end{proof}

\section{Applications of the finite-holonomy theorem}
\label{holonomyapplications}

This section collects global and local consequences of Theorem~\ref{theoremA}. We
first treat compact leaves, then develop the holonomy-continuation argument near
invariant rational curves, and finally apply it to non-dicritical surface singularities
and transverse surface sections.

\subsection{Compact leaves and global integrability}

\begin{proof}[Proof of Theorem~\ref{compactleafalgebraic}]
Let \(L\) be a compact leaf of \(\mathcal F\). Since \(L\) is contained in the
regular locus of the foliation and \(\mathcal F\) has corank one, it is a smooth
compact analytic hypersurface of \(X\). By Chow's theorem, \(L\) is an algebraic
\(\mathcal F\)-invariant prime divisor.

Theorem~\ref{theoremA} shows that the holonomy group of \(L\) is finite. Since
\(L\) is compact and disjoint from \(\operatorname{Sing}(\mathcal F)\), the
foliation is regular on an analytic neighbourhood of \(L\). The holomorphic Reeb
local stability theorem therefore gives a saturated neighbourhood of \(L\) containing
infinitely many compact leaves. By the Darboux--Jouanolou theorem
\cite[Theorem~2.3]{PereiraSpicer2020} (See also \cite{Jouanolou1978}), the existence of infinitely many
invariant algebraic hypersurfaces implies that \(\mathcal F\) admits a
rational first integral, i.e. algebraically integrable.
\end{proof}

\begin{remark}\label{remarkcompactleafstructure}
Theorem \ref{theoremA} gives more precise information on the
geometry of the compact leaf. Let \(\Gamma\) be the holonomy group of
\(L\). Since \(\Gamma\) is finite, it is cyclic and analytically
linearizable. Its linear part is, up to the usual inverse convention,
the monodromy of the flat normal bundle \(N_{L/X}\). Hence
\(N_{L/X}\) is torsion; let \(k\) denote its order.

The holonomy is in particular abelian and formally linearizable.
Therefore
\cite[Proposition~6.1\textup{(2)}]
{ClaudonLorayPereiraTouzet2018}
implies that \(L\) has infinite Ueda type. It follows from
\cite[Theorem~1.3]{ClaudonLorayPereiraTouzet2018}
that \(kL\) is a fibre of a holomorphic fibration
\(X\to C\) onto a curve.

Finally,
\cite[Proposition~6.3]{ClaudonLorayPereiraTouzet2018}
shows that there exist a projective manifold \(Z\) and a generically
finite morphism \(\pi:Z\to X\) such that the pull-back foliation
\(\pi^*\mathcal F\) is defined by a closed rational \(1\)-form.
\end{remark}

\subsection{First integrals near invariant rational curves}

\begin{lemma}\label{primitivefirstintegrals}
Let $f,g:(S,p)\to(\mathbb C,0)$ be primitive holomorphic first integrals, in the
sense of Definition~\ref{definitionfirstintegral}, of the same rank-one foliation germ
on a normal surface. Then there is a one-variable biholomorphic germ $\varphi$ such
that
\(
g=\varphi\circ f.
\)
\end{lemma}

\begin{proof}
After shrinking representatives, the connected components of the fibres of either
first integral are the separatrixes. In particular, $g$ is constant on every connected
component of a fibre of $f$.

Take the local Stein factorisation of $f$,
\[
(S,p)\xrightarrow{h}(T,0)\xrightarrow{\alpha}(\mathbb C,0),
\qquad f=\alpha\circ h,
\]
where the general fibres of $h$ are connected and $\alpha$ is finite. The normal curve
germ $(T,0)$ is smooth, so we identify it with $(\mathbb C,0)$. Since $g$ is constant
on the connected fibres of $h$, it descends to a holomorphic germ
$\beta:(T,0)\to(\mathbb C,0)$ and
\(
g=\beta\circ h.
\)

The primitivity of $f$ means that its general fibre is connected. If the finite map
$\alpha$ had degree greater than one, a general fibre of $f$ would be the disjoint union
of several general fibres of $h$. Hence $\alpha$ has degree one and is biholomorphic.
The same argument, applied to $g$, shows that $\beta$ is biholomorphic. Therefore
\(
g
=
\bigl(\beta\circ\alpha^{-1}\bigr)\circ f.
\)
Taking $\varphi=\beta\circ\alpha^{-1}$ proves the assertion; in particular,
$\varphi'(0)\neq0$.
\end{proof}

\begin{proposition}\label{firstintegralnearcomponent}
Let \(Y\) be a normal surface which is smooth in a neighbourhood of an
irreducible curve \(C\simeq\mathbb P^1\). Let \(\mathcal G\) be a rank-one
foliation on \(Y\), and suppose that \(C\) is \(\mathcal G\)-invariant. Set
\(
S:=C\cap\operatorname{Sing}(\mathcal G), C^\circ:=C\setminus S.
\)

Fix a point \(q\in C^\circ\) and a transverse disc \(\Sigma\) through
\(q\). Denote by
\(\Gamma\subset\operatorname{Diff}(\Sigma,q)\) the holonomy group of
\(C^\circ\). For every \(p\in S\), let
\(\Gamma_p\subseteq\Gamma\) be the local holonomy subgroup at \(p\),
transported to \(\Sigma\) as in
Definition~\ref{definitionlocalholonomy}.

Assume that \(\Gamma\) is finite and that, for every \(p\in S\), there are
holomorphic coordinates \(u_p,v_p\) centred at \(p\), with
\(C=(u_p=0)\), and coprime positive integers \(a_p,b_p\) such that
\(
F_p:=u_p^{a_p}v_p^{b_p}
\)
is a primitive holomorphic first integral of \(\mathcal G\) at \(p\).

Then \(\mathcal G\) admits a holomorphic first integral on a neighbourhood
of \(C\).

Moreover, if \(\Gamma_p=\Gamma\) for some \(p\in S\), then the first
integral may be chosen primitive at \(p\). In particular, if
\(\#S\leq2\), it may be chosen primitive at every point of \(S\).
\end{proposition}

\begin{proof}
Since \(Y\) is smooth near \(C\), the singular locus of the saturated
rank-one foliation \(\mathcal G\) is discrete there. Hence
\(S=C\cap\operatorname{Sing}(\mathcal G)\) is finite.

\smallskip
\noindent\textbf{Step 1: Construction near \(C^\circ\).}

Choose \(q\in C^\circ\) and a small transverse disc \(\Sigma\) through
\(q\). We regard the holonomy group as a finite subgroup
\(\Gamma\subset\operatorname{Diff}(\Sigma,q)\). Since a finite subgroup
of \(\operatorname{Diff}(\mathbb C,0)\) is cyclic and holomorphically
linearizable, there is a coordinate \(t\) on \(\Sigma\) in which
\(\Gamma=\{t\mapsto\zeta t\mid\zeta^N=1\}\), where \(N:=|\Gamma|\).
The germ \(H_\Sigma:=t^N\) is invariant under \(\Gamma\).

Let \(r\in C^\circ\), choose a transverse disc \(\Sigma_r\) through
\(r\), and let \(\delta\) be a path in \(C^\circ\) from \(q\) to \(r\).
If \(h_\delta:(\Sigma,q)\to(\Sigma_r,r)\) denotes the corresponding
holonomy map, set
\(H_{\Sigma_r}:=H_\Sigma\circ h_\delta^{-1}\).

This definition is independent of \(\delta\). Indeed, if \(\delta'\) is
another path from \(q\) to \(r\), then
\(g:=h_{\delta'}^{-1}\circ h_\delta\) belongs to \(\Gamma\). Since
\(h_\delta^{-1}=g^{-1}\circ h_{\delta'}^{-1}\), the
\(\Gamma\)-invariance of \(H_\Sigma\) gives
\(H_\Sigma\circ h_\delta^{-1}
 =H_\Sigma\circ h_{\delta'}^{-1}\).

Extending these transverse germs constantly along the plaques in local
flow boxes gives a holomorphic first integral \(H^\circ\) on a
neighbourhood \(U^\circ\) of \(C^\circ\).

\smallskip
\noindent\textbf{Step 2: Local holonomy near a point of \(S\).}

Fix \(p\in S\). In the coordinates appearing in the statement, we have
\(C=(u_p=0)\) and
\(F_p=u_p^{a_p}v_p^{b_p}\), where
\(\gcd(a_p,b_p)=1\). Choose a sufficiently small \(v_0\neq0\), put
\(q_p=(0,v_0)\in C^\circ\), and let
\(\Sigma_p=(v_p=v_0)\) be the transverse disc through \(q_p\). Its
transverse coordinate is \(u_p\). After shrinking, we may assume that
\(\Sigma_p\subset U^\circ\), and we write
\(H_{\Sigma_p}:=H^\circ|_{\Sigma_p}\).

To compute the local holonomy around \(p\), consider the positively
oriented loop \(v_p(s)=v_0e^{2\pi i s}\), for \(0\leq s\leq1\).
The lift starting at \((u_0,v_0)\in\Sigma_p\) remains in the same level
of \(F_p\), and therefore satisfies
\(u_p(s)^{a_p}v_p(s)^{b_p}=u_0^{a_p}v_0^{b_p}\). The solution with
\(u_p(0)=u_0\) is
\[u_p(s)=u_0e^{-2\pi i b_ps/a_p}.\] Thus the local holonomy germ is
\(h_p(u_p)=\xi_pu_p\), where
\(\xi_p=e^{-2\pi i b_p/a_p}\). Since \(a_p\) and \(b_p\) are coprime,
\(\xi_p\) is a primitive \(a_p\)-th root of unity.

Because \(H^\circ\) is a first integral, its restriction to
\(\Sigma_p\) is invariant under this local holonomy. Thus
\(H_{\Sigma_p}(\xi_pu_p)=H_{\Sigma_p}(u_p)\). It follows from the Taylor
expansion that only powers divisible by \(a_p\) occur, so there is a
one-variable holomorphic germ \(\chi_p\) such that
\[
H_{\Sigma_p}(u_p)=\chi_p(u_p^{a_p}).
\]

\smallskip
\noindent\textbf{Step 3: Extension across singularities.}

On \(\Sigma_p\), the primitive first integral satisfies
\(F_p|_{\Sigma_p}=v_0^{b_p}u_p^{a_p}\). Define
\(\varphi_p(s):=\chi_p(v_0^{-b_p}s)\). Then
\(\varphi_p(F_p|_{\Sigma_p})=H_{\Sigma_p}\), and hence
\(H_p:=\varphi_p\circ F_p\) is a holomorphic first integral on a
sufficiently small neighbourhood \(V_p\) of \(p\), with
\(H_p|_{\Sigma_p}=H^\circ|_{\Sigma_p}\).

By construction,
\[
H_p|_{\Sigma_p}=H^\circ|_{\Sigma_p}.
\]
We justify that this equality extends to a neighbourhood of
\(\Sigma_p\). Although the coordinates \((u_p,v_p)\) need not themselves
be flow-box coordinates near \(q_p\), the prescribed disc
\(\Sigma_p=(v_p=v_0)\) is transverse to \(\mathcal G\) at \(q_p\).
Hence, after shrinking, there is a flow box
\(B\subset U^\circ\cap V_p\) adapted to \(\Sigma_p\).

More precisely, if \(X_p\) is a local holomorphic vector field generating
\(\mathcal G\) near \(q_p\), its local flow identifies a product
\[
\Delta\times\Sigma_p \longrightarrow B,
\qquad
(s,z)\longmapsto\Phi_s(z),
\]
biholomorphically with \(B\). Let
\(\pi_p:B\to\Sigma_p\) be the projection along the plaques. Since both
\(H^\circ\) and \(H_p\) are first integrals, they are constant along the
plaques, and hence
\[
H^\circ
=
H^\circ|_{\Sigma_p}\circ\pi_p,
\qquad
H_p
=
H_p|_{\Sigma_p}\circ\pi_p.
\]
Their restrictions to \(\Sigma_p\) agree, so
\(H^\circ=H_p\) on \(B\).

After shrinking \(V_p\), we may assume that
\(U^\circ\cap V_p\) is connected. The identity theorem then gives
\(H^\circ=H_p\) on the whole overlap. Thus the two functions glue to a
holomorphic first integral extending \(H^\circ\) across \(p\).

\smallskip
\noindent\textbf{Step 4: Primitivity of the extended function.}

Choose a path \(\gamma_p\subset C^\circ\) from \(q\) to \(q_p\), and
let \(h_{\gamma_p}:(\Sigma,q)\to(\Sigma_p,q_p)\) be its holonomy map.
By construction,
\(H_{\Sigma_p}=H_\Sigma\circ h_{\gamma_p}^{-1}\). Since
\(h_{\gamma_p}\) is biholomorphic, it preserves multiplicity. Therefore
\[
N
=\operatorname{mult}_{q_p}(H_{\Sigma_p})
=a_p\,\operatorname{mult}_0(\chi_p),
\qquad
\operatorname{mult}_0(\varphi_p)=\frac{N}{a_p}.
\]

Since \(F_p\) is primitive, the first integral
\(H_p=\varphi_p\circ F_p\) is primitive at \(p\) if and only if
\(\varphi_p\) is biholomorphic, equivalently if and only if \(N=a_p\).
If \(\Gamma_p\subseteq\Gamma\) denotes the transported local holonomy
subgroup, then \(|\Gamma_p|=a_p\) and \(|\Gamma|=N\). Thus \(H_p\) is
primitive precisely when \(\Gamma_p=\Gamma\).

Finally, suppose that \(\#S\leq2\). If \(S=\varnothing\), there is
nothing to prove. If \(S=\{p\}\), then
\(C^\circ\simeq\mathbb C\) is simply connected, so both \(\Gamma\) and
\(\Gamma_p\) are trivial. If \(S=\{p_1,p_2\}\), then
\(C^\circ\simeq\mathbb C^*\), and a small loop around either puncture
generates \(\pi_1(C^\circ)\), up to inversion. Hence
\(\Gamma_{p_i}=\Gamma\) for \(i=1,2\), and the constructed first integral
is primitive at every point of \(S\).
\end{proof}
\begin{proposition}[Irreducible exceptional curve case]\label{irreducibleexceptionalfibre}
Let \(X\) be a normal klt surface, let \(x\in X\) be a closed point, and
let \(\mathcal F\) be a non-dicritical foliation which is analytically
\(p\)-closed at \(x\). Suppose that there exists a proper birational
morphism
\(
\pi:Y\longrightarrow X
\)
which is an isomorphism over \(X\setminus\{x\}\), such that \(Y\) is smooth,
the reduced exceptional fibre
\(
E:=\pi^{-1}(x)_{\mathrm{red}}
\)
is irreducible, and the strict transform
\(
\mathcal G:=\pi^{-1}\mathcal F
\)
has only reduced singularities along \(E\). Then \(\mathcal F\) admits a
holomorphic first integral at \(x\).
\end{proposition}

\begin{proof}
Choose a local algebraic \(p\)-closed model. The analytic isomorphism of normal
surface germs lifts to the minimal resolutions and to the subsequent point blow-ups,
so we may perform the given modification on the local model. By
Lemma~\ref{pclosebirational}, the strict transform \(\mathcal G\) is
\(p\)-closed for almost all primes and analytically \(p\)-closed at every point of
\(E\).

Since \(X\) is klt and \(E\) is irreducible, the curve \(E\) is a smooth
rational curve. By non-dicriticality, it is \(\mathcal G\)-invariant. Its singularities along \(E\) are
reduced, and Lemma~\ref{pclosednotsaddle} excludes saddle-nodes. Hence
Theorem~\ref{canonicalcases}\textup{(1)} gives a primitive monomial holomorphic
first integral at every point of
\(E\cap\operatorname{Sing}(\mathcal G)\).

By Theorem~\ref{theoremA}, the holonomy group of
\(
E\setminus\operatorname{Sing}(\mathcal G)
\)
is finite. Proposition~\ref{firstintegralnearcomponent} therefore gives a
holomorphic first integral
\(
H:U\longrightarrow\mathbb C
\)
on a neighbourhood \(U\) of \(E\).

The function \(H\) is constant on the invariant curve \(E\). Since \(X\)
is normal and \(\pi\) is proper birational,
\(
\pi_*\mathcal O_Y=\mathcal O_X.
\)
After shrinking around \(x\), the function \(H\) therefore descends to a
holomorphic function near \(x\). Away from \(x\), this function induces
\(\mathcal F\), and by saturation it induces \(\mathcal F\) on the whole
germ. Thus it is a holomorphic first integral at \(x\).
\end{proof}

\subsection{Non-dicritical surface applications}

We now introduce the reduction condition used in the reducible exceptional case and
prove the local surface application of Theorem~\ref{theoremA}.

\begin{definition}[Reduction and non-branching reduction]
\label{definitionnonbranching}
Let \((X,\mathcal F,x)\) be a non-dicritical foliated surface germ. By a
\emph{reduction of singularities over \(x\)} we mean a proper birational morphism
\(
\pi:(Y,\mathcal G)\longrightarrow(X,\mathcal F)
\)
which is an isomorphism away from \(x\), such that \(Y\) is smooth, the reduced
exceptional curve
\(
E:=\pi^{-1}(x)_{\mathrm{red}}
\)
is a simple normal crossing divisor, and the transformed foliation
\(\mathcal G:=\pi^{-1}\mathcal F\) has only reduced singularities along \(E\).

Writing \(E=E_1\cup\cdots\cup E_m\), we call the reduction
\emph{non-branching} if
\(
\#\bigl(E_i\cap\operatorname{Sing}(\mathcal G)\bigr)\le2
\)
for every irreducible component \(E_i\).
\end{definition}

\begin{proof}[Proof of Corollary~\ref{nondicriticalsurfaceapplication}]
Choose a local algebraic \(p\)-closed model. A reduction of a surface foliation is
obtained from the minimal resolution followed by point blow-ups; the analytic
isomorphism of germs lifts to these modifications. We may therefore perform the given
reduction on the local algebraic model and keep the same notation. By
Lemma~\ref{pclosebirational}, the transformed foliation \(\mathcal G\) is
\(p\)-closed for almost all primes and analytically \(p\)-closed at every point of
\(E\). If the exceptional fibre is irreducible, the
conclusion follows from Proposition~\ref{irreducibleexceptionalfibre}; hence, in
the remainder of the proof, we may assume that the reduction is non-branching and
that the exceptional fibre has more than one irreducible component.
Since \(X\) is klt, every irreducible component of
\(E\) is a smooth rational curve, and the intersection graph of the exceptional
components is connected and contains no cycles; see
\cite[Proposition~4.18]{KollarMori1998}. By non-dicriticality, every exceptional
component is \(\mathcal G\)-invariant.

Its singularities along \(E\) are reduced,
and Lemma~\ref{pclosednotsaddle} excludes saddle-nodes. Thus they are
non-degenerate reduced, hence canonical, surface singularities.
Theorem~\ref{canonicalcases}\textup{(1)} therefore supplies a primitive monomial first
integral at each singular point of \(\mathcal G\) on \(E\).

Theorem~\ref{theoremA} consequently gives finite holonomy along the regular part of
every irreducible component of \(E\).

Write
\(
E=E_1\cup\cdots\cup E_m.
\)
For every \(i\), Proposition~\ref{firstintegralnearcomponent} gives a holomorphic
first integral
\(
H_i:U_i\longrightarrow\mathbb C
\)
on a neighbourhood of \(E_i\). Since \(E_i\) contains at most two singular points,
the final assertion of that proposition allows \(H_i\) to be chosen primitive at
every singular point on \(E_i\).

It remains to combine the functions constructed near the different exceptional
components. If \(E_i\) and \(E_j\) meet at a point \(p\), choose a primitive local
first integral \(F_p\) there. By Lemma~\ref{primitivefirstintegrals},
\[
H_i=\alpha_{i,p}\circ F_p,
\qquad
H_j=\alpha_{j,p}\circ F_p
\]
for one-variable biholomorphic germs \(\alpha_{i,p}\) and \(\alpha_{j,p}\).

Reorder the components so that, for every \(k\ge2\), the curve \(E_k\) meets
\(E_1\cup\cdots\cup E_{k-1}\) at exactly one point \(p_k\), lying on a unique
component \(E_{j(k)}\) with \(j(k)<k\). Such an ordering exists because the
intersection graph is a finite connected graph without cycles.

Keep \(H_1\) fixed. Suppose inductively that the first integrals near
\(E_1,\ldots,E_{k-1}\) have been reparametrised so that they agree on all their
overlaps. Near \(p_k=E_k\cap E_{j(k)}\), choose a one-variable biholomorphic germ
\(\theta_k\) such that
\(
H_{j(k)}=\theta_k\circ H_k.
\)
Replacing \(H_k\) by \(\theta_k\circ H_k\) makes the two functions agree near
\(p_k\). Since \(E_k\) meets the previously treated components only at \(p_k\),
this replacement does not affect any earlier compatibility. Proceeding inductively
and shrinking the neighbourhoods \(U_i\) so that two of them meet only near the
corresponding intersection point, the functions glue to a holomorphic first integral
\(
H:U\longrightarrow\mathbb C
\)
on a neighbourhood of the whole exceptional fibre.

Every component of \(E\) is invariant, so \(H\) is constant on each component and,
since \(E\) is connected, on all of \(E\). Again
\(\pi_*\mathcal O_Y=\mathcal O_X\), so \(H\) descends to a holomorphic first
integral of \(\mathcal F\) near \(x\).
\end{proof}

\subsection{Simple singularities in arbitrary dimension}

We now apply the finite-holonomy theorem to general surface sections of simple
singularities. The first proposition describes the surface singularities obtained from a
simple singularity of arbitrary dimensional type.

\begin{proposition}\label{simplesurfacesection}
Let \(X\) be a smooth variety and let \(\mathcal F\) be a corank one foliation with a
simple singularity of dimensional type \(r\) at a closed point \(x\in X\).
\begin{enumerate}[label=\textup{(\arabic*)}]
\item If \(r=2\), then the restriction of \(\mathcal F\) to a sufficiently general
smooth surface through \(x\) has a reduced singularity at \(x\).
\item If \(r\geq3\), then, for a sufficiently general smooth surface \(S\) through
\(x\), the blow-up
\(
\pi:\widetilde S\longrightarrow S
\)
at \(x\) has an invariant exceptional curve \(E\simeq\mathbb P^1\), and the
transformed foliation has exactly \(r\) reduced singularities on \(E\).
\end{enumerate}
\end{proposition}

\begin{proof}
The simple normal form is a priori formal. The calculation below,
however, uses only a finite jet of the defining form and of the surface
section. In type~\textup{(2)}, we also retain the first nonzero term of
\(\psi\) whenever it is needed. Choose an integer \(N\) larger than all
the orders occurring in the calculation, and truncate the formal
coordinate change and the formal unit relating the defining form to its
normal form at order \(N\). Since the linear part of the formal
coordinate change is invertible, its truncation defines an analytic
coordinate system near \(x\). Thus we may carry out the calculation in
analytic coordinates in which the defining form has the same relevant
finite jet as the formal normal form.

Assume first that \(r=2\). Choose a sufficiently general smooth surface
\(S\) through \(x\) such that
\(
u:=x_1|_S,
v:=x_2|_S
\)
are local coordinates at \(x\).

For a singularity of type~\textup{(1)}, the linear part of the
restricted defining form is
\(
\lambda_1v\,du+\lambda_2u\,dv.
\)
A local vector field generating the restricted foliation therefore has
eigenvalues \(\lambda_2\) and \(-\lambda_1\). They are both nonzero. If
their quotient were a positive rational number, there would be
positive integers \(a,b\) such that
\(a\lambda_1+b\lambda_2=0,
\)
contrary to the non-resonance condition.

For a singularity of type~\textup{(2)}, if \(k=2\), the linear part is
\(
p_1v\,du+p_2u\,dv,
\)
and the two eigenvalues are \(p_2\) and \(-p_1\), whose quotient is
negative. If \(k=1\), the restricted form has the shape
\[
p_1v\,du
+
\alpha_2u\psi(u^{p_1})\,dv
+
\text{higher-order terms}.
\]
Here \(\psi\not\equiv0\); otherwise the defining form would not have
dimensional type two after saturation. The corresponding vector field
has one zero eigenvalue and one nonzero eigenvalue, and hence the
singularity is a saddle-node. Thus the restricted singularity is
reduced in every case, proving~\textup{(1)}.

Assume now that \(r\geq3\). Let \(z,w\) be coordinates on a sufficiently
general smooth surface \(S\) through \(x\). Write the linear parts of
the restrictions of the active coordinates as
\[
\ell_i(z,w)=a_i z+b_i w,
\qquad i=1,\ldots,r.
\]
The generality of \(S\) ensures that the \(\ell_i\) are nonzero and
pairwise nonproportional. These conditions are open conditions on the
tangent plane \(T_xS\).

For a singularity of type~\textup{(1)}, put
\(\rho_i:=\lambda_i\). For one of type~\textup{(2)}, put
\[
\rho_i=
\begin{cases}
p_i,&1\leq i\leq k,\\
0,&k<i\leq r.
\end{cases}
\]
In either case, the lowest-degree homogeneous part of the restricted
defining form is
\[
\eta_{r-1}
=
\ell_1\cdots\ell_r
\sum_{i=1}^r
\rho_i\frac{d\ell_i}{\ell_i}.
\]
Indeed, in type~\textup{(2)} the terms involving \(\psi\) have strictly
larger degree because \(\psi(0)=0\). Set
\(
R:=\sum_{i=1}^r\rho_i.
\)
In type~\textup{(1)}, the non-resonance condition, applied to
\((1,\ldots,1)\), gives \(R\neq0\). In type~\textup{(2)},
\[
R=p_1+\cdots+p_k>0.
\]

Blow up the origin of \(S\). After making a linear change of the
coordinates \(z,w\), we may suppose that none of the \(r\) directions
\(\ell_i=0\) is the point at infinity in the chart
\(
z=s,
w=st.
\)
Put
\(
L_i(t):=a_i+b_i t,
P(t):=\prod_{i=1}^rL_i(t),
\)
and
\[
Q(t)
:=
\sum_{i=1}^r
\rho_i b_i
\prod_{j\neq i}L_j(t).
\]
After dividing the pull-back of the restricted form by its common
exceptional factor \(s^{r-1}\), we obtain
\[
\widetilde\eta
=
\bigl(RP(t)+sA(s,t)\bigr)\,ds
+
s\bigl(Q(t)+sB(s,t)\bigr)\,dt
\tag{\(\dagger\)}
\]
for suitable holomorphic functions \(A\) and \(B\). In particular, the
coefficient of \(dt\) is divisible by \(s\), so the exceptional curve
\(E=(s=0)\) is invariant.

We next check that \((\dagger)\) is saturated along \(E\). Let \(t_i\)
be the zero of \(L_i\). Since the \(\ell_i\) are pairwise
nonproportional, \(t_i\) is a simple zero of \(P\), and
\[
Q(t_i)=\rho_iP'(t_i).
\tag{\(\ddagger\)}
\]

If \(\rho_i\neq0\), then \(Q(t_i)\neq0\). Thus the two coefficients in
\((\dagger)\) have no common factor near \((0,t_i)\).

It remains to consider the case \(\rho_i=0\). This occurs only in
type~\textup{(2)}, with \(i>k\). In particular, the non-resonance
condition implies that \(\alpha_i\neq0\). Moreover,
\(\psi\not\equiv0\); otherwise the common factor
\(x_{k+1}\cdots x_r\) could be removed and the dimensional type would be
at most \(k<r\). Write
\[
\psi(T)=cT^\nu+O(T^{\nu+1}),
\qquad
c\neq0,
\qquad
T=x_1^{p_1}\cdots x_k^{p_k}.
\]
Near the point \((0,t_i)\), put \(\tau=t-t_i\). Since \(i>k\), all the
functions \(L_1,\ldots,L_k\) are nonzero at \(t_i\), and therefore
\[
T=s^R u(s,\tau),
\qquad
u(0,0)\neq0.
\]
The summand
\(
\alpha_i\psi(T)
\prod_{j\neq i}x_j\,dx_i
\)
of the original defining form contributes, after division by
\(s^{r-1}\), a term
\(
\kappa\,s^{\nu R+1}\,d\tau
\)
modulo terms divisible by \(\tau\) or of higher \(s\)-order, where
\(\kappa\neq0\). On the other hand, the lowest-order \(ds\)-coefficient
is
\(
RP'(t_i)\tau
\)
modulo \((s,\tau^2)\). Hence the \(ds\)-coefficient is not divisible by
\(s\), while the \(d\tau\)-coefficient is not divisible by \(\tau\).
The two coefficients therefore have no common factor at
\((0,t_i)\).

Consequently, \((\dagger)\) is a saturated local generator of the
transformed foliation along \(E\). Its restriction to \(E\) is
\(
\widetilde\eta|_E=RP(t)\,ds.
\)
Since \(R\neq0\), the singular points of the transformed foliation on
\(E\) are precisely the zeros of \(P\). Projectively, these zeros are
the \(r\) distinct points of
\(
E=\mathbb P(T_xS)\simeq\mathbb P^1
\)
defined by \(\ell_i=0\). Thus the transformed foliation has exactly
\(r\) singular points on \(E\).

It remains to prove that these singularities are reduced. Fix \(i\)
and use \(\tau=t-t_i\) as a local coordinate on \(E\). Since
\[
P(t)=P'(t_i)\tau+O(\tau^2)
\]
and \((\ddagger)\) holds, a local vector field generating the
transformed foliation has, up to multiplication by a nonzero scalar, a
triangular linear part whose eigenvalues are
\(
\rho_i\) and
\(
-R.
\)

In type~\textup{(1)}, suppose that
\(-R/\rho_i\in\mathbb Q_{>0}\). Then there are positive integers
\(m,n\) such that
\(
nR+m\rho_i=0.
\)
Since \(R=\lambda_1+\cdots+\lambda_r\) and
\(\rho_i=\lambda_i\), this gives a nontrivial relation among the
\(\lambda_j\) with nonnegative integral coefficients, contradicting
the non-resonance condition. Hence every singularity is
non-degenerate and reduced.

In type~\textup{(2)}, if \(i\leq k\), then
\(\rho_i=p_i>0\), whereas \(-R<0\). The quotient of the two eigenvalues
is therefore negative, so the singularity is non-degenerate and
reduced. If \(i>k\), then \(\rho_i=0\) and the other eigenvalue
\(-R\) is nonzero. The singularity is therefore a saddle-node, and in
particular is reduced.

This proves~\textup{(2)}.
\end{proof}

\begin{proof}[Proof of Corollary~\ref{simpleintegrability}]
Let \(r\) be the dimensional type of the simple singularity. Choose a local
\(p\)-closed model
\(
(X_0,\mathcal F_0,x_0)
\)
of \((X,\mathcal F,x)\), together with an analytic isomorphism of the two
foliated germs. Since \((X,x)\) is smooth, after shrinking \(X_0\) around
\(x_0\) we may assume that \(X_0\) is smooth. Choose a sufficiently general
smooth algebraic complete-intersection surface
\(
x_0\in S_0\subset X_0
\)
which satisfies Proposition~\ref{simplesurfacesection} and the transversality
hypotheses of Lemma~\ref{MM80}. Let \(S\subset X\) be the corresponding
analytic surface germ. By the proof of Lemma~\ref{hypersurfacesectionpclosed}, applied
successively, \(\mathcal F_0|_{S_0}\) is \(p\)-closed for almost all primes.
Consequently, \(\mathcal F|_S\) is analytically \(p\)-closed at \(x\).

Suppose first that \(r=2\). By Proposition~\ref{simplesurfacesection}, the
restricted foliation \(\mathcal F_0|_{S_0}\) has a reduced singularity at
\(x_0\), in particular canonical. By
Theorem~\ref{canonicalcases}\textup{(1)}, \(\mathcal F_0|_{S_0}\) admits a
holomorphic first integral at \(x_0\). Transporting it through the analytic
isomorphism gives a holomorphic first integral of \(\mathcal F|_S\) at \(x\).

Assume now that \(r\geq3\). Let
\(
\pi:\widetilde S_0\longrightarrow S_0
\)
be the blow-up of \(x_0\), and set
\(
\mathcal G_0:=\pi^{-1}(\mathcal F_0|_{S_0}).
\)
By Proposition~\ref{simplesurfacesection}, the exceptional curve
\(E\simeq\mathbb P^1\) is \(\mathcal G_0\)-invariant and
\(\mathcal G_0\) has exactly \(r\) reduced singularities on \(E\). By
Lemma~\ref{pclosebirational}, \(\mathcal G_0\) is analytically \(p\)-closed at every point of \(E\). 
Every singularity of \(\mathcal G_0\) on \(E\) is reduced. So Corollary \ref{nondicriticalsurfaceapplication} gives a local holomorphic first integral of $\FF|_{S}$. In both cases, Lemma~\ref{MM80} lifts the surface first integral to a
holomorphic first integral of \(\mathcal F\) at \(x\).
\end{proof}

\section{Formal first integrals for rank-one foliations}\label{formalfirstintegral}
In contrast with the corank one case, formal integrability need not imply holomorphic
integrability for rank-one foliations, even when the singular locus is isolated; see
\cite{SilvaEtAl2024}. We nevertheless obtain a complete set of formal first integrals
under analytic $p$-closedness and canonical singularities.

\begin{theorem}\label{rankone1}
Let $X$ be a smooth variety of dimension $n$, and let $\mathcal F$ be a rank one
foliation on $X$ with canonical singularities. Let $x\in X$ be a closed point. If
$\mathcal F$ is analytically $p$-closed at $x$, then there exists a formal morphism
\[
\hat f:\widehat X\longrightarrow \widehat{\mathbb C^{\,n-1}}_0
\]
such that
\[
\ker d\hat f=\widehat{\mathcal F},
\]
where $\widehat X$ and $\widehat{\mathcal F}$ denote completion at $x$, and
$\widehat{\mathbb C^{\,n-1}}_0$ denotes the formal completion of $\mathbb C^{n-1}$
at the origin.
\end{theorem}

\begin{proof}
If $\mathcal F$ is regular at $x$, then the claim follows from the formal Frobenius
theorem. Thus we may assume that $x\in \operatorname{Sing}(\mathcal F)$.

Choose a local $p$-closed model of $(X,\mathcal F,x)$. By Proposition~\ref{mcq08}, after
passing to the formal completion at $x$, there exist formal coordinates
$x_1,\ldots,x_n$ and integers $\lambda_1,\ldots,\lambda_n$ such that
$\widehat{\mathcal F}$ is generated by
\[
V=\sum_{i=1}^n \lambda_i x_i\frac{\partial}{\partial x_i}.
\]
Here we allow some of the $\lambda_i$ to be zero; equivalently, if the normal form in
Proposition~\ref{mcq08} involves fewer than $n$ coordinates, we extend it to a full
system of formal coordinates and assign weight zero to the remaining variables.

Since $\mathcal F$ has canonical singularities, the nonzero weights are not all of the
same sign. Replacing $V$ by $-V$ if necessary and reordering the coordinates, we may
write
\[
\lambda_i>0 \quad (i\in P),\qquad
\lambda_j<0 \quad (j\in N),\qquad
\lambda_k=0 \quad (k\in Z),
\]
where $P,N,Z$ form a partition of $\{1,\ldots,n\}$ and both $P$ and $N$ are nonempty.
Choose indices \(p\in P\) and \(q\in N\).

We now define \(n-1\) formal first integrals. For every \(k\in Z\), set
\(
F_k=x_k.
\)
For every \(j\in N\), set
\[
F_j=x_p^{-\lambda_j}x_j^{\lambda_p},
\]
and for every \(i\in P\setminus\{p\}\), set
\[
G_i=x_i^{-\lambda_q}x_q^{\lambda_i}.
\]
All exponents which occur here are positive integers. The collection
\[
\{F_k\}_{k\in Z},\qquad \{F_j\}_{j\in N},\qquad \{G_i\}_{i\in P\setminus\{p\}}
\]
has cardinality
\[
|Z|+|N|+(|P|-1)=n-1.
\]
Let \(\hat f:\widehat X\to \widehat{\mathbb C^{\,n-1}}_0\) be the formal morphism whose
components are these functions.

By construction, each component of \(\hat f\) is annihilated by \(V\). Hence
\(
\widehat{\mathcal F}\subseteq \ker d\hat f.
\)
It remains to prove the reverse inclusion. Let
\[
D=\sum_{r=1}^n A_r\frac{\partial}{\partial x_r}
\]
be a formal vector field such that \(D\) annihilates all components of \(\hat f\). From
\(D(F_k)=D(x_k)=0\) for \(k\in Z\), we get \(A_k=0\) for all \(k\in Z\).

For \(j\in N\), the equation \(D(F_j)=0\) is equivalent, after dividing by the common
monomial factor, to
\[
(-\lambda_j)x_jA_p+\lambda_p x_pA_j=0.
\]
Taking \(j=q\), we see that \(x_p\) divides \(A_p\). Thus we can write
\(A_p=\lambda_p x_p h\) for some formal power series \(h\). Substituting this into the
same equation gives
\(
A_j=\lambda_j x_j h
\)
for every \(j\in N\).

Now let \(i\in P\setminus\{p\}\). The equation \(D(G_i)=0\) gives
\[
(-\lambda_q)x_qA_i+\lambda_i x_iA_q=0.
\]
Since \(A_q=\lambda_q x_qh\), it follows that
\(
A_i=\lambda_i x_i h.
\)
Therefore
\[
D=h\sum_{r=1}^n \lambda_r x_r\frac{\partial}{\partial x_r}=hV.
\]
Thus every formal vector field tangent to the fibres of \(\hat f\) is a multiple of
\(V\). Hence \(\ker d\hat f=\widehat{\mathcal F}\), as required.
\end{proof}

\section{The Bott partial connection}\label{bottconnectionsection}

We first recall the Bott partial connection and its dual description. We then explain
how holomorphic first integrals produce flat full connections and prove
Theorem~\ref{bottdimtwo}.

\subsection{The regular local model}

\begin{definition}\label{bottconnection}
Let \(X\) be a smooth variety and let \(\mathcal F\) be a regular foliation on \(X\).
Let \(T_{\mathcal F}\subset T_X\) be its tangent bundle and let
\(N_{\mathcal F}:=T_X/T_{\mathcal F}\). The Bott partial connection on
\(N_{\mathcal F}\) is the map
\[
\nabla^{\mathrm{Bott}}:T_{\mathcal F}\times N_{\mathcal F}\to N_{\mathcal F}
\]
defined by
\[
\nabla^{\mathrm{Bott}}_v([w])=[v,w]\bmod T_{\mathcal F},
\]
where \(v\) is a local section of \(T_{\mathcal F}\), and \([w]\) denotes the image of a
local vector field \(w\) in \(N_{\mathcal F}\).
\end{definition}

\begin{remark}
The definition is independent of the choice of the representative \(w\). Indeed, if
\(w\) is replaced by \(w+u\), with \(u\in T_{\mathcal F}\), then
\([v,u]\in T_{\mathcal F}\) since \(T_{\mathcal F}\) is closed under the Lie bracket.
\end{remark}

\begin{proposition}\label{dualBott}
Let \(X\) be a smooth variety and let \(\mathcal F\) be a regular foliation on \(X\).
Identifying
\[
N_{\mathcal F}^*
=
\{\alpha\in\Omega_X^1\mid \alpha(u)=0 \text{ for all } u\in T_{\mathcal F}\},
\]
the dual Bott partial connection on \(N_{\mathcal F}^*\) is given by
\[
\nabla^{\mathrm{Bott},*}_v\alpha=\mathcal N_v\alpha
\]
for every local section \(v\in T_{\mathcal F}\) and every local section
\(\alpha\in N_{\mathcal F}^*\).
\end{proposition}

\begin{proof}
First, \(\mathcal N_v\alpha\) is again a section of \(N_{\mathcal F}^*\). Indeed, for
\(u\in T_{\mathcal F}\), we have
\[
(\mathcal N_v\alpha)(u)=v(\alpha(u))-\alpha([v,u]).
\]
Since \(\alpha(u)=0\) and \([v,u]\in T_{\mathcal F}\), it follows that
\((\mathcal L_v\alpha)(u)=0\).

Let \([w]\) be a local section of \(N_{\mathcal F}\), represented by a local vector field
\(w\). By definition,
\[
\nabla^{\mathrm{Bott}}_v[w]=[v,w]\bmod T_{\mathcal F}.
\]
The dual connection is characterized by
\[
(\nabla^{\mathrm{Bott},*}_v\alpha)([w])
=
v\bigl(\alpha([w])\bigr)
-
\alpha\bigl(\nabla^{\mathrm{Bott}}_v[w]\bigr).
\]
Since \(\alpha\) annihilates \(T_{\mathcal F}\), this gives
\[
(\nabla^{\mathrm{Bott},*}_v\alpha)([w])
=
v(\alpha(w))-\alpha([v,w])
=
(\mathcal L_v\alpha)(w).
\]
Thus \(\nabla^{\mathrm{Bott},*}_v\alpha=\mathcal L_v\alpha\).
\end{proof}

\begin{proposition}[Regular local model for Conjecture~\ref{conjbottconnection}]
\label{prop:regular-local-model}
Let \(X\) be a smooth variety, and let
\(\mathcal F\) be a regular holomorphic foliation of codimension \(q\), with tangent
bundle \(T_{\mathcal F}\subset T_X\).
Let \(x\in X\). After replacing $X$ by a sufficiently small analytic neighbourhood of $x$, the
following conditions are equivalent.

\begin{enumerate}[label=\textup{(\arabic*)}]
\item The foliation \(\mathcal F\) admits holomorphic first integrals
\(f^1,\ldots,f^q\), with \(df^1,\ldots,df^q\) linearly independent, such that
\(T_{\mathcal F}=\bigcap_i\ker df^i\).

\item There exists a flat holomorphic connection
\[
\nabla:N_{\mathcal F}\longrightarrow N_{\mathcal F}\otimes\Omega_X^1
\]
whose restriction along \(T_{\mathcal F}\) is the Bott partial connection, and
such that \(N_{\mathcal F}^*\) admits a frame of closed holomorphic \(1\)-forms
which are parallel with respect to the dual connection \(\nabla^*\).
\end{enumerate}

Moreover, in any local conormal frame
\(\omega=(\omega^1,\ldots,\omega^q)\), written as a row vector, the above
connection can be described by a matrix \(\eta\) of holomorphic \(1\)-forms
satisfying
\[
d\omega=-\omega\wedge\eta,
\qquad
\nabla^*\omega=\omega\eta,
\qquad
d\eta+\eta\wedge\eta=0.
\]
\end{proposition}

\begin{proof}
Assume first that \((1)\) holds. Put \(\alpha^i=df^i\). Then
\(\alpha=(\alpha^1,\ldots,\alpha^q)\) is a frame of \(N_{\mathcal F}^*\)
consisting of closed holomorphic \(1\)-forms. Define \(\nabla^*\) by declaring
this frame to be parallel. This connection is flat, and its dual gives a flat
connection \(\nabla\) on \(N_{\mathcal F}\).

We check that \(\nabla\) restricts to the Bott partial connection. Equivalently,
we check the dual statement. If \(V\in T_{\mathcal F}\), then
\(i_V\alpha^i=0\) and \(d\alpha^i=0\), hence Cartan's formula gives
\(\mathcal L_V\alpha^i=0\). Since also \(\nabla^*_V\alpha^i=0\) by construction,
\(\nabla^*\) agrees with the dual Bott partial connection along
\(T_{\mathcal F}\). Thus \((2)\) holds.

Conversely, assume that \((2)\) holds, and let
\(\alpha^1,\ldots,\alpha^q\) be a closed \(\nabla^*\)-parallel frame of
\(N_{\mathcal F}^*\). After shrinking, the holomorphic Poincaré lemma gives
holomorphic functions \(f^1,\ldots,f^q\) with \(\alpha^i=df^i\). Since the
\(\alpha^i\) form a frame of the conormal bundle, their common kernel is
\(T_{\mathcal F}\). Hence \(T_{\mathcal F}=\bigcap_i\ker df^i\), so the
\(f^i\) are local holomorphic first integrals.

It remains to record the matrix form. Let
\(\omega=(\omega^1,\ldots,\omega^q)\) be any local frame of \(N_{\mathcal F}^*\).
Since \(\mathcal F\) is integrable, the ideal generated by the \(\omega^i\) is
closed under exterior differentiation. Therefore there exists a matrix
\(\eta\) of holomorphic \(1\)-forms such that
\[
d\omega=-\omega\wedge\eta.
\]
Define \(\nabla^*\) in this frame by \(\nabla^*\omega=\omega\eta\). If
\(V\in T_{\mathcal F}\), then \(i_V\omega=0\), and Cartan's formula gives
\[
\mathcal L_V\omega=i_Vd\omega
=i_V(-\omega\wedge\eta)
=\omega\,\eta(V)
=\nabla^*_V\omega.
\]
Thus this connection extends the dual Bott partial connection.

If moreover \(d\eta+\eta\wedge\eta=0\), then locally one can solve
\(dg=-\eta g\) with \(g\colon U\to \operatorname{GL}_q(\mathbb C)\). For
\(\alpha=\omega g\), we have \(\nabla^*\alpha=0\), and using
\(d\omega=-\omega\wedge\eta\) and \(dg=-\eta g\), one gets
\[
d\alpha=d(\omega g)=(d\omega)g-\omega\wedge dg
=-\omega\wedge(\eta g+dg)=0.
\]
Hence \(\alpha\) is a closed parallel conormal frame. This is precisely the
closed parallel frame appearing in \((2)\).
\end{proof}

\subsection{Meromorphic extensions and proof of Theorem~\ref{bottdimtwo}}

We next record the construction in the presence of a degeneracy locus.

\begin{lemma}\label{lem:first-integrals-bott-connection}
Let $V$ be a smooth analytic variety and let $\mathcal F$ be a regular
foliation of corank $q$ on $V$. Suppose that there are holomorphic functions
$F_1,\ldots,F_q$ such that, on the dense open subset
\[
V^\circ:=\{dF_1\wedge\cdots\wedge dF_q\neq0\},
\]
the forms $dF_1,\ldots,dF_q$ define $\mathcal F$. Then there is a unique flat
meromorphic connection $\nabla^*$ on $N_{\mathcal F}^*$ whose restriction to
$V^\circ$ satisfies
\[
\nabla^*(dF_i)=0
\qquad (1\le i\le q).
\]
Its dual connection extends the Bott partial connection, and the $dF_i$ form a closed
parallel conormal frame on $V^\circ$.
\end{lemma}

\begin{proof}
On $V^\circ$, the forms $dF_1,\ldots,dF_q$ are a global frame of
$N_{\mathcal F}^*$. Declaring this frame horizontal defines a global flat connection
there.

Let $W\subset V$ be an open set on which $N_{\mathcal F}^*$ has a holomorphic
frame $\varepsilon=(\varepsilon_1,\ldots,\varepsilon_q)$, written as a row vector.
Write
\[
\alpha:=(dF_1,\ldots,dF_q)=\varepsilon A
\]
for a holomorphic matrix $A$. On $W\cap V^\circ$, write
$\nabla^*\varepsilon=\varepsilon\Omega$. The condition $\nabla^*\alpha=0$ gives
\[
\Omega=-dA\,A^{-1}.
\]
The matrix $A^{-1}$ is meromorphic along $\det A=0$, so $\Omega$ extends
meromorphically to $W$.

These local matrices satisfy the usual gauge-transformation law. Indeed, on an
overlap, if $\varepsilon'=\varepsilon B$, then
\[
\Omega'=B^{-1}\Omega B+B^{-1}dB.
\]
Hence the locally defined meromorphic connections glue uniquely. Equivalently, they
all restrict to the same globally defined connection on the dense open set $V^\circ$.
This is ordinary sheaf-theoretic gluing of connection matrices, not holonomy
continuation along paths.

Since $\Omega=-dA\,A^{-1}$ is a pure gauge, it satisfies
\[
d\Omega+\Omega\wedge\Omega=0.
\]
Thus the meromorphic connection is flat. Finally,
if $v\in T_{\mathcal F}$, then
\[
\mathcal L_v(dF_i)=d(v(F_i))=0=\nabla_v^*(dF_i).
\]
By Proposition~\ref{dualBott}, the restriction of $\nabla^*$ along
$T_{\mathcal F}$ is the dual Bott connection on the dense open set, and hence
meromorphically on $V$. Dualising proves the assertion for $N_{\mathcal F}$.
\end{proof}

\begin{proof}[Proof of Theorem~\ref{bottdimtwo}]
Assume first that the canonical-singularity case of
Conjecture~\ref{conjectureholomorphiccanonical} holds. Let $x\in X$ and choose an
analytic neighbourhood $U$ on which a single holomorphic map
\[
F=(F_1,\ldots,F_q):U\longrightarrow\mathbb C^q
\]
induces $\mathcal F$. Put $V=U\setminus\operatorname{Sing}(\mathcal F)$. The
foliation is regular on $V$, so Lemma~\ref{lem:first-integrals-bott-connection}
produces the required flat meromorphic full connection on $N_{\mathcal F}|_V$ and
the required closed parallel frame on a dense open subset. This proves that
Conjecture~\ref{conjectureholomorphiccanonical}\textup{(1)} implies
Conjecture~\ref{conjbottconnection}.

If $\dim X=2$, the required holomorphic first integral is supplied by
Theorem~\ref{canonicalcases}\textup{(1)}. Hence
Conjecture~\ref{conjbottconnection} holds in dimension two.
\end{proof}

\begin{remark}\label{remark-bott-no-homotopy}
No homotopy or holonomy gluing is used in Theorem~\ref{bottdimtwo}. The conjectural
holomorphic integrability statement supplies one map $F$ on the whole neighbourhood
$U$, rather than unrelated first integrals on an open cover. Its differentials form a
global conormal frame on a dense open subset. The only ``gluing'' is the standard
compatibility of the meromorphic connection matrices $-dA\,A^{-1}$ under changes of
holomorphic frame; it is expressed by the gauge-transformation formula
\[
\Omega'=B^{-1}\Omega B+B^{-1}dB.
\]
Thus this gluing is sheaf-theoretic and does not involve analytic continuation along
paths.
\end{remark}

\begin{remark}[Explicit surface model]
In dimension two the connection can be written explicitly. Suppose that, by
Theorem~\ref{canonicalcases}\textup{(1)}, there are analytic coordinates \(x,y\) such that the foliation
is generated by
\[
v=nx\partial_x-my\partial_y,
\]
where \(m,n>0\). Let \(U_x=\{x\neq0\}\) and \(U_y=\{y\neq0\}\). On \(U_x\), the normal
bundle \(N_{\mathcal F}\) is generated by \(e_y=[\partial_y]\), while on \(U_y\) it is
generated by \(e_x=[\partial_x]\). On the overlap \(U_x\cap U_y\), the relation
\([v]=0\) gives
\[
nx[\partial_x]-my[\partial_y]=0,
\]
and hence
\[
e_y=\frac{nx}{my}e_x.
\]

Define a meromorphic connection \(\nabla\) on
\(N_{\mathcal F}|_{U\setminus\operatorname{Sing}(\mathcal F)}\) by
\[
\nabla e_y=
\left(
m\frac{dx}{x}+(n-1)\frac{dy}{y}
\right)\otimes e_y
\quad \text{on } U_x,
\]
and
\[
\nabla e_x=
\left(
(m-1)\frac{dx}{x}+n\frac{dy}{y}
\right)\otimes e_x
\quad \text{on } U_y.
\]
These two expressions agree on \(U_x\cap U_y\), since
\(e_y=(nx/my)e_x\). Therefore \(\nabla\) is well defined. It is flat, and its dual
connection satisfies
\[
\nabla^*(d(x^m y^n))=0.
\]
Thus \(\nabla\) is the meromorphic flat connection extending the Bott partial connection
constructed in Lemma~\ref{lem:first-integrals-bott-connection}.
\end{remark}

\printbibliography

@incollection{Andre2004,
  author    = {Andr\'{e}, Yves},
  title     = {Sur la conjecture des $p$-courbures de Grothendieck--Katz et un probl\`eme de Dwork},
  booktitle = {Geometric Aspects of Dwork Theory},
  publisher = {de Gruyter},
  address   = {Berlin},
  year      = {2004},
  pages     = {55--112}
}

@article{Bost2001,
  author  = {Bost, Jean-Benoit},
  title   = {Algebraic Leaves of Algebraic Foliations over Number Fields},
  journal = {Publications Math\'{e}matiques de l'IH\'{E}S},
  volume  = {93},
  year    = {2001},
  pages   = {161--221},
  doi     = {10.1007/s10240-001-8191-3}
}

@article{Cano2004,
  author  = {Cano, Felipe},
  title   = {Reduction of the Singularities of Codimension One Singular Foliations in Dimension Three},
  journal = {Annals of Mathematics},
  volume  = {160},
  number  = {3},
  year    = {2004},
  pages   = {907--1011}
}

@article{CasciniSpicer2021,
  author  = {Cascini, Paolo and Spicer, Calum},
  title   = {MMP for Co-rank One Foliations on Threefolds},
  journal = {Inventiones Mathematicae},
  volume  = {225},
  number  = {2},
  year    = {2021},
  pages   = {603--690}
}

@article{CerveauLinsNeto2008,
  author  = {Cerveau, Dominique and Lins Neto, Alcides},
  title   = {Frobenius Theorem for Foliations on Singular Varieties},
  journal = {Bulletin of the Brazilian Mathematical Society},
  volume  = {39},
  year    = {2008},
  pages   = {447--469}
}

@article{Druel2025,
  author  = {Druel, St\'{e}phane},
  title   = {$A$-analyticity of Separatrices of Foliations},
  journal = {Beitr\"age zur Algebra und Geometrie},
  year    = {2025}
}

@unpublished{EkedahlShepherdBarronTaylor1999,
  author = {Ekedahl, Torsten and Shepherd-Barron, Nicholas I. and Taylor, Richard},
  title  = {A Conjecture on the Existence of Compact Leaves of Algebraic Foliations},
  note   = {Preprint},
  year   = {1999}
}

@article{Katz1972,
  author  = {Katz, Nicholas M.},
  title   = {Algebraic Solutions of Differential Equations ($p$-Curvature and the Hodge Filtration)},
  journal = {Inventiones Mathematicae},
  volume  = {18},
  year    = {1972},
  pages   = {1--118}
}

@book{KollarMori1998,
  author    = {Koll\'{a}r, J\'{a}nos and Mori, Shigefumi},
  title     = {Birational Geometry of Algebraic Varieties},
  series    = {Cambridge Tracts in Mathematics},
  volume    = {134},
  publisher = {Cambridge University Press},
  address   = {Cambridge},
  year      = {1998}
}

@article{LorayPereiraTouzet2018,
  author  = {Loray, Frank and Pereira, Jorge Vit\'{o}rio and Touzet, Fr\'{e}d\'{e}ric},
  title   = {Singular Foliations with Trivial Canonical Class},
  journal = {Inventiones Mathematicae},
  volume  = {213},
  year    = {2018},
  pages   = {1327--1380}
}

@article{Malgrange1976,
  author  = {Malgrange, Bernard},
  title   = {Frobenius avec singularit\'{e}s. I. Codimension un},
  journal = {Publications Math\'{e}matiques de l'IH\'{E}S},
  volume  = {46},
  year    = {1976},
  pages   = {163--173}
}

@article{MatteiMoussu1980,
  author  = {Mattei, Jean-Fran\c{c}ois and Moussu, Robert},
  title   = {Holonomie et int\'{e}grales premi\`eres},
  journal = {Annales Scientifiques de l'\'{E}cole Normale Sup\'{e}rieure},
  series  = {4},
  volume  = {13},
  number  = {4},
  year    = {1980},
  pages   = {469--523}
}

@article{McQuillan2008,
  author  = {McQuillan, Michael},
  title   = {Canonical Models of Foliations},
  journal = {Pure and Applied Mathematics Quarterly},
  volume  = {4},
  number  = {3},
  year    = {2008},
  pages   = {877--1012}
}

@article{McQuillanPanazzolo2013,
  author  = {McQuillan, Michael and Panazzolo, Daniel},
  title   = {Almost \'{E}tale Resolution of Foliations},
  journal = {Journal of Differential Geometry},
  volume  = {95},
  number  = {2},
  year    = {2013},
  pages   = {279--319}
}

@book{MiyaokaPeternell1997,
  author    = {Miyaoka, Yoichi and Peternell, Thomas},
  title     = {Geometry of Higher Dimensional Algebraic Varieties},
  series    = {DMV Seminar},
  volume    = {26},
  publisher = {Birkh\"auser},
  year      = {1997}
}

@article{Moussu1998,
  author  = {Moussu, Robert},
  title   = {Sur l'existence d'int\'{e}grales premi\`eres holomorphes},
  journal = {Annali della Scuola Normale Superiore di Pisa, Classe di Scienze},
  series  = {4},
  volume  = {26},
  number  = {4},
  year    = {1998},
  pages   = {709--717}
}

@article{Seidenberg1968,
  author  = {Seidenberg, Abraham},
  title   = {Reduction of Singularities of the Differential Equation $A\,dy=B\,dx$},
  journal = {American Journal of Mathematics},
  volume  = {90},
  year    = {1968},
  pages   = {248--269}
}

@article{SilvaEtAl2024,
  author  = {da Silva, A. B. and others},
  title   = {A Remark on First Integrals of Vector Fields},
  journal = {Publicacions Matem\`atiques},
  volume  = {68},
  number  = {1},
  year    = {2024},
  pages   = {103--109}
}

@article{Spicer2020,
  author  = {Spicer, Calum},
  title   = {Higher-Dimensional Foliated Mori Theory},
  journal = {Compositio Mathematica},
  volume  = {156},
  number  = {1},
  year    = {2020},
  pages   = {1--38}
}

@article{SpicerSvaldi2021,
  author  = {Spicer, Calum and Svaldi, Roberto},
  title   = {Local and Global Applications of the Minimal Model Program for Co-rank One Foliations},
  journal = {Journal of the European Mathematical Society},
  volume  = {24},
  number  = {11},
  year    = {2022},
  pages   = {3969--4025}
}

@article{Xu2025,
  author  = {Xu, Yuting},
  title   = {$p$-Integrability},
  journal = {arXiv preprint},
  year    = {2025},
  eprint  = {2507.22056},
  archiveprefix = {arXiv}
}

@article{Jouanolou1978,
  author  = {Jouanolou, Jean-Pierre},
  title   = {Hypersurfaces solutions d'une {\'e}quation de {Pfaff} analytique},
  journal = {Mathematische Annalen},
  volume  = {232},
  number  = {3},
  pages   = {239--245},
  year    = {1978},
  doi     = {10.1007/BF01351428}
}

@article{PereiraSpicer2020,
  author  = {Pereira, Jorge Vit{\'o}rio and Spicer, Calum},
  title   = {Hypersurfaces Quasi-Invariant by Codimension One Foliations},
  journal = {Mathematische Annalen},
  volume  = {378},
  number  = {1--2},
  pages   = {613--635},
  year    = {2020},
  doi     = {10.1007/s00208-019-01833-4}
}

@article{Katz1982,
  author  = {Katz, Nicholas M.},
  title   = {A Conjecture in the Arithmetic Theory of Differential Equations},
  journal = {Bulletin de la Soci{\'e}t{\'e} Math{\'e}matique de France},
  volume  = {110},
  number  = {2},
  pages   = {203--239},
  year    = {1982},
  doi     = {10.24033/bsmf.1960}
}

@incollection{ChudnovskyChudnovsky1985,
  author    = {Chudnovsky, David V. and Chudnovsky, Gregory V.},
  title     = {Applications of Pad\'{e} Approximations to the Grothendieck Conjecture on Linear Differential Equations},
  booktitle = {Number Theory (New York, 1983--84)},
  series    = {Lecture Notes in Mathematics},
  volume    = {1135},
  publisher = {Springer},
  address   = {Berlin},
  pages     = {52--100},
  year      = {1985},
  doi       = {10.1007/BFb0074601}
}

@article{ClaudonLorayPereiraTouzet2018,
  author  = {Claudon, Beno{\^i}t and Loray, Frank and Pereira, Jorge Vit{\'o}rio and Touzet, Fr{\'e}d{\'e}ric},
  title   = {Compact Leaves of Codimension One Holomorphic Foliations on Projective Manifolds},
  journal = {Annales Scientifiques de l'{\'E}cole Normale Sup{\'e}rieure},
  series  = {4},
  volume  = {51},
  number  = {6},
  pages   = {1457--1506},
  year    = {2018},
  doi     = {10.24033/asens.2379}
}

\end{document}